\documentclass[11pt,onecolumn]{IEEEtran}
\usepackage{subfigure,amsmath,amsfonts,amssymb,color,graphicx, algorithmic,verbatim,mathrsfs,multirow,setspace,cite}

\newtheorem{theorem}{Theorem}
\newtheorem{lemma}{Lemma}
\newtheorem{claim}{Claim}
\newtheorem{proposition}{Proposition}

\newtheorem{assumption}{Assumption}

\newenvironment{proof}[1][Proof]{\begin{trivlist}
\item[\hskip \labelsep {\bfseries #1}]}{\end{trivlist}}

\newcommand{\qed}{\nobreak \ifvmode \relax \else
      \ifdim\lastskip<1.5em \hskip-\lastskip
      \hskip1.5em plus0em minus0.5em \fi \nobreak
      \vrule height0.75em width0.5em depth0.25em\fi}

      \def\an#1{{\color{red} #1}}
      \def\sml#1{{\color{blue} #1}}
      \def\dist{\mathrm{dist}}

\def\a{\alpha}
\def\e{\epsilon}
\def\Xc{\mathcal{X}}
\def\Fc{\mathcal{F}}
\def\EXP#1{{\mathsf{E}\left[#1 \right]}}
\def\la{\langle}
\def\ra{\rangle}

\title{\bf Distributed Random Projection Algorithm for Convex Optimization}
\author{Soomin Lee and Angelia Nedi\'{c}
\thanks{
S.~Lee is with the Electrical and Computer Engineering Department and A.~Nedi\'c is with
Industrial and Enterprise Systems Engineering Department at
the University of Illinois, Urbana, IL 61801;
e-mails: \text{\{lee203, angelia\}}@illinois.edu.
A. Nedi\'c gratefully acknowledges the support of this work by the NSF
through the awards CMMI 07-42538 and CCF 11-11342, and the Navy support under the grant N00014-12-1-0998.}}
\begin{document}
\maketitle
\begin{abstract}
\noindent
Random projection algorithm is of interest for constrained optimization when the constraint set is not known in advance or the projection operation on the whole constraint set is computationally prohibitive.
This paper presents a \textit{distributed random projection} (DRP) algorithm for fully distributed constrained convex optimization problems that can be used by multiple agents connected over a time-varying network,
where each agent has its own objective function and its own constrained set.
With reasonable assumptions, we prove that the iterates of all agents converge to the same point
in the optimal set almost surely.
In addition, we consider a variant of the method that uses a mini-batch of
consecutive random projections and establish its convergence in almost sure sense.
Experiments on distributed support vector machines demonstrate fast convergence of the algorithm.
It actually shows that the number of iteration required until convergence is much smaller
than scanning over all training samples just once.
\end{abstract}

\section{Introduction}

A number of problems that arise in sensor, wireless \textit{ad hoc} and peer-to-peer networks can be formulated as convex constrained minimization problems~\cite{rabbat,ram_info,johansson,con03}.
The goal of the agents connected over such networks is to cooperatively solve the following optimization problem:
\begin{align}\label{eqn:prob}
\min_{x} {}& f(x) = \sum_{i=1}^m f_i(x)\qquad \text{s.t. } x \in \mathcal{X} \triangleq \bigcap_{i=1}^m \mathcal{X}_i,
\end{align}
where each $f_i: \mathbb{R}^d \rightarrow \mathbb{R}$ is a convex function,
representing the local objective function of agent $i$,
and each $\mathcal{X}_i \subseteq \mathbb{R}^d$ is a closed convex set,
representing the local constraint set of agent $i$.
The complete problem information is not available at a single location. This is because i) there is no central node that facilitates computation and communication
and ii) it is often not possible for one agent to keep all the objective and constraint components due to memory, computational power, or privacy constraints.
In addition, the network topology itself may change with time due to agent mobility or link failures.
Therefore, an optimization algorithm for solving such problems must be distributed and robust, so that
each agent exchanges its information only with its immediate neighbors and the algorithm has to be adaptive to the changes in the network topology.

In this paper, we propose a \textit{distributed random projection} (DRP) algorithm
for problem (\ref{eqn:prob}), where the constraint set is defined as the intersection of
finitely many simple convex constraints.
That is, $\Xc_i = \bigcap_{j\in I_i}\Xc_i^j$, where $I_i$ is a finite\footnote{The finiteness of $I_i$
is not really crucial. The developed results also apply to the case when the index sets $I_i$
are infinite.}
(a formal definition of $I_i$ is in Section \ref{sec:algo}).
In our algorithm, each agent $i$ maintains its own iterate sequence $\{x_i(k)\}$.
At each iteration, each agent calculates weighted average of the received iterates (from its neighbors)
and its own iterate, adjusts the iterate by using gradient information of its local objective function $f_i$ and projects onto a constraint component that is selected randomly from its local constraint set $\Xc_i$.
The projections are performed locally by each agent based on the random observations of
the local constraint components. In particular, agent $i$ observes a constraint component
$\Xc_i^{\Omega_i(k)}$ at time $k$, where $\Omega_i(k) \in I_i$ is a random variable.

Our primary interest is in the case when the whole constraint set $\Xc_i$ for an agent $i$ is not known in advance, but its component is revealed through random realizations $\Xc_i^{\Omega_i(k)}$.
For example, in collaborative filtering for recommender systems,
user data is huge and distributed over multiple machines.
Users frequently change and update their preferences in real time
so the constraint set of this problem is usually not known in advance.
Another case of interest is when the whole constraint set $\Xc_i$
is known in advance but it has a huge number of components.
For example, in text classification problems, model parameters
are trained based on hundred thousands or more text samples and
each sample constitutes a constraint component (usually a halfspace)~\cite{Joachims:2006}.
In such a case, the projection operation on the whole constraint set $\Xc_i$ is computationally prohibitive
if any of the traditional (sub)gradient projection methods are used.
In Section \ref{sec:DrSVM},
we will experiment with Support Vector Machines to classify three text data sets.

Problem (\ref{eqn:prob}) can be solved by the incremental or
the Markov incremental algorithm, and the distributed subgradient algorithm.
In the Markov incremental algorithm studied in~\cite{Johansson2007,Ram:2009},
the agents maintain a single estimate sequence that is sequentially updated by one agent at a time.
When an agent receives the estimate, it updates the estimate using its local objective function and passes it to a randomly selected neighbor.
The update order is driven by a time inhomogeneous Markov chain (as the network topology is time varying).
Whereas in the distributed subgradient algorithms, each agent maintains its own estimate. It communicates the estimate with its neighbors and updates it using the local objective and constraint information. Algorithms of this type requires a consensus over all agents for convergence.
However, in some distributed problems it is important that each agent maintains a good estimate at all times.
For example, in a distributed online learning, each node is expected to perform in real time.
Our DRP algorithm is in the distributed subgradient algorithm category.

The related distributed optimization literature
includes~\cite{Sayed2008a,Sayed2010a,Sayed2006a,AN2007,SayedLopes2007,AN2009,ANquan2008,
Lobel2011},
which are concerned with convex but unconstrained problems,
and~\cite{Ram2010,Ram2012,Duchi2012} where constrained problems are considered.
The most relevant to the work in this paper
are~\cite{NO2010,SN2010cdc,SN2011,LobelOF2011} where, as in the DRP algorithm,
the constraint set is also distributed across agents and each agent handles its own constraint set only.
In~\cite{NO2010}, the convergence analysis is done for a special case when the network is completely connected. The work in~\cite{SN2010cdc,SN2011} extends the algorithm and its analysis
to a more general network including the presence of noisy links,
while~\cite{LobelOF2011} extends it to a general Markovian network model.
Unlike~\cite{NO2010} and~\cite{SN2010cdc}, where each agent can perform projections on
its entire constraint set, this paper addresses the case when such projections are not possible or computationally prohibitive.
Related to this work are also the distributed algorithms for estimation and inference problems
that have been proposed and studied
by Sayed et al.~\cite{Sayed2012,ChenSayed2012,TuSayed2012a,TuSayed2012b}.
On a much broader scale, the work in this paper is related to the literature on the consensus problem,
where each agent starts from an initial value 
and ends by converging to
a value common to all agents (see for example\cite{con01,con02,con03,con04,con05}).


The contribution of this paper is mainly in two directions.
First, we propose a novel distributed optimization algorithm that is based on local
communications of agents' estimates in a network and a gradient descent with random projections.
Second, we study the convergence of the algorithm and its variant using a mini-batch
of random projections.
To the best of our knowledge, there is no previous work on distributed optimization algorithms that
utilize random projections. Gradient and subgradient random projection algorithms
for {\it centralized} (not distributed) convex problems have been proposed in~\cite{AN2011}.
Finding probabilistic feasible solutions through random sampling of constraints for optimization problems with uncertain constraint sets have been proposed in \cite{tempo2009,Calafiore:2010}.
Also, the related work is the (centralized) random projection method for a special class of convex feasibility problems, which has been proposed and studied by Polyak \cite{Polyak2001409}.

The rest of the paper is organized as follows.
In Section \ref{sec:algo},
we introduce the problem of interest, formally describe our algorithm and state assumptions on the problem and network.
In Section \ref{sec:prelim}, we state some results from the literature that we use in the convergence analysis.
In Section \ref{sec:lem}, we derive two important results that will play crucial roles
in the convergence analysis.
In Section \ref{sec:conv}, we study the almost sure convergence property of our DRP algorithm.
We provide an extension of the algorithm to a variant that uses a mini-batch of random projections
and we state a convergence result for this extension in Section~\ref{sec:mini-batches}.
As an application of our DRP algorithm and its mini-batch variant, in Section~\ref{sec:DrSVM},
we introduce a linear SVM formulation, discuss how to apply the algorithm, and present some experimental results on binary text classification tasks.
Section \ref{sec:con} contains concluding remarks and future directions.

\noindent\textbf{Notation}
A vector is viewed as a column. We write $x^T$ to denote the transpose of a vector $x$. The scalar product of two vectors $x$ and $y$ is $\langle x, y \rangle$.
We use a subscript $i$ to denote an agent $i$.
An index $k$ with parentheses is devoted to represent a time.
For example, $x_i(k)$ is the iterate of an agent $i$ at time $k$.
We use $\|x\|$ to denote the standard Euclidean norm. We write $\text{dist}(x,\mathcal{X})$
for the distance of a vector $x$ from a closed convex set $\mathcal{X}$, \textit{i.e.},
$\text{dist}(x,\mathcal{X}) = \min_{v\in\mathcal{X}}\|v-x\|$. We use $\mathsf{\Pi}_{\mathcal{X}}[x]$
for the projection of a vector $x$ on the set $\mathcal{X}$, \textit{i.e.}, $\mathsf{\Pi}_{\mathcal{X}}[x] = \arg\min_{v\in\mathcal{X}}\|v-x\|^2$.
We use $\mathsf{Pr}\{Z\}$ and $\mathsf{E}[Z]$ to denote
the probability and the expectation of a random variable $Z$.
We abbreviate \textit{almost surely} and \textit{independent and identically distributed} as \textit{a.s.} and \textit{iid}, respectively.

\section{Problem Set-up, Algorithm and Assumptions \label{sec:algo}}
\subsection{Optimization over a Network}
We consider a constrained convex optimization problem~(\ref{eqn:prob}) that is distributed over a network
of $m$ agents, indexed by $V = \{1,\ldots,m\}$.
The function $f_i$ and the constraint set $\Xc_i$ in~\eqref{eqn:prob} are private information of agent $i$
(not shared with any other agent).
Collectively, the agents are responsible for solving problem~\eqref{eqn:prob}.

We are interested in the case when each constraint set $\Xc_i$ is the intersection
of finitely many closed convex sets.
Without loss of generality, let $\Xc$ be the intersection of $n$ closed convex sets.
Let $I = \{1,\ldots,n\}$ be the index set, and let $I_i$, $i \in V$, be a partition of $I$
(i.e., $I = \bigcup_{i=1}^m I_i$ and $I_i \cap I_j = \emptyset$ for $i \neq j$) such that
each $I_i$ is associated with the local constraint set $\mathcal{X}_i$ of agent $i$, i.e.,
\[\Xc_i=\cap_{j\in I_i} \Xc^j_i\qquad\hbox{for a finite index set $I_i$},\]
where the superscript is used to identify a component set.
Each component set $\Xc_i^j$ is assumed to be a "simple set" for the projection operation.
Examples of such simple sets include a halfspace
$\Xc_i^j=\{x\in \mathbb{R}^d\mid \langle a, x\rangle \leq b\}$,
a box $\Xc_i^j=\{x\in \mathbb{R}^d\mid \alpha \le x \le \beta\}$ (the inequality is component-wise)
and a ball  $\Xc_i^j=\{x\in \mathbb{R}^d\mid \|x-v\|\le r \}$, where $a,\alpha, \beta, v \in \mathbb{R}^d$ and $b,r \in \mathbb{R}$.
In such cases, the projection on the set $\Xc_i$ can be complex, especially when the number
of components is large, while the projection on each component $\Xc_i^j$ has a closed form expression.

We use the following assumption for the functions $f_i$ and the sets $\Xc_i^j$.

\begin{assumption} \label{assume:f} Let the following conditions hold:
\begin{enumerate}
\item[(a)] The sets $\mathcal{X}_i^j$, $j\in I_i$ are closed and convex for every $i \in V$.
\item[(b)] Each function $f_i:\mathbb{R}^d\to\mathbb{R}$ is convex.
\item[(c)] The functions $f_i$, $i \in V$, are differentiable and have
\textit{Lipschitz gradients} with a constant $L$ over $\mathbb{R}^d$,
\[
\|\nabla f_i(x) - \nabla f_i(y)\| \leq L \|x-y\| \quad \text{for all } x, y \in \mathbb{R}^d.
\]
\item[(d)] The gradients $\nabla f_i(x)$, $i\in V$ are bounded over the set $\mathcal{X}$, i.e.,
there exists a constant $G_f$ such that
\[
\|\nabla f_i(x)\|  \leq G_f \quad \text{for all $x \in\mathcal{X}$ and all $i\in V$}.
\]
\end{enumerate}
\end{assumption}

When each $f_i$ has Lipschitz gradients with a constant $L_i$, Assumption~\ref{assume:f}(c) is satisfied
with $L=\max_{i\in V} L_i$. Further note that
Assumption~\ref{assume:f}(d) is satisfied, for example, when $\mathcal{X}$ is compact.

As mentioned earlier, the agents are collectively responsible for solving problem~\eqref{eqn:prob},
without sharing their private knowledge of individual objective functions $f_i$
and the constrained sets $\Xc_i$. To accommodate such a task, the agents are assumed to form a
network, wherein each agent communicates its iterates to its local neighbors.
More specifically, at each time $k$, the network topology is represented by a directed
graph $G(k) = (V, E(k))$, where $E(k) \subseteq V \times V$. A link $(i,j)\in E(k)$ indicates that
agent $i$ has received information from agent $j$ at time $k$.
We let $N_i(k)$ denote the set of agents who send information to agent $i$,
i.e., $N_i(k)=\{j\in V\mid (i,j)\in E(k)\}$.
We assume that $i \in N_i(k)$ for all $i \in V$ and for all~$k$.

\subsection{Distributed Random Projection Algorithm (DRP)}
To solve the problem (\ref{eqn:prob}) with distributed information access, we propose an iterative gradient method with random projections.
Let $x_i(k) \in \mathbb{R}^d$ denote the estimate of agent $i$ at time $k$.
At time $k$, each agent sends the estimate to its neighbors (represented by the graph $(V,E(k))$.
Upon receiving the estimates $x_j(k)$ from its neighbors $j\in N_i(k)$, each agent $i$
updates according to the following two steps:
\begin{subequations}
\begin{align}
v_i(k) = {}&  \sum_{j\in N_i(k)} w_{ij}(k) x_j(k) \label{eqn:algo3}\\
x_i(k+1) = {}& \mathsf{\Pi}_{\mathcal{X}_i^{\Omega_i(k)}} \left[ v_i(k) - \alpha_k \nabla f_i(v_i(k))\right], \label{eqn:algo4}
\end{align}
\end{subequations}
where $\alpha_k > 0$ is a stepsize at time $k$ and $x_i(0) \in \mathbb{R}^d$ is an initial estimate of
agent $i$ (which can be random).

In the above, relation (\ref{eqn:algo3}) captures an information mixing step, while (\ref{eqn:algo4})
captures a local minimization and feasibility update step using a random projection.
In (\ref{eqn:algo3}), the iterate $v_i(k)$ is a weighted average of agent $i$'s estimate and
the estimates received from its neighbors $j\in N_i(k)$. Specifically, $w_{ij}(k)\ge0$ is a weight
that agent $i$ places on the estimate $x_j(k)$ received from a neighbor $j\in N_i(k)$ at time $k$,
where the total weight sum is 1, i.e., $\sum_{j\in N_i(k)} w_{ij}(k)=1$ for each agent $i$.
The step (\ref{eqn:algo3}) can be equivalently represented as
\begin{align}\label{eqn:algo3mix}
v_i(k) =  \sum_{j=1}^m [W(k)]_{ij} x_j(k)
\end{align}
by letting $w_{ij}(k) = 0$ for whenever $j \not\in N_i(k)$,
and using $[W]_{ij}$ to denote the $(i,j)$th entry of a matrix $W$.

In (\ref{eqn:algo4}), agent $i$ adjusts
the average $v_i(k)$ along the negative gradient direction of its local objective $f_i$.
At time $k$, agent $i$ also observes a random realization of its local constraint component set
$\mathcal{X}_i^{\Omega_i(k)}$.
To reduce the feasibility violation, it projects its current estimate on this set.
The random variable $\Omega_i(k)$ takes values in the index set $I_i$ at all times $k$.
In this way, instead of projecting onto the whole local constraint set $\mathcal{X}_i$, agent
$i$ projects only on a component set $\mathcal{X}_i^{\Omega_i(k)}$ which is randomly selected at time $k$.
Note that the updated estimate $x_i(k+1)$ may not lie in $\mathcal{X}_i$ since $\mathcal{X}_i \subset \mathcal{X}_i^{\Omega_i(k)}$.

Through the updates (\ref{eqn:algo3}) and (\ref{eqn:algo4}),
agents combine their information and consider their own optimization problem
of minimizing $f_i$ over the set $\Xc_i$.
There is neither a central node governing the whole process nor additional constraints enforcing consistency.
Nevertheless, with this simple update rule, our algorithm finds the optimal solution and all agents eventually arrive at a common optimal solution (all $x_i(k)$ converge to some $x^* \in \mathcal{X}^*$, as
shown in Section \ref{sec:conv}).

Note that algorithm (\ref{eqn:algo3})-(\ref{eqn:algo4}) is similar to the distributed projected
subgradient algorithm in~\cite{NO2010} except for the randomization over the components of the set
$\Xc_i$ in (\ref{eqn:algo4}).
At each iteration of the algorithm in~\cite{NO2010},
a projection is performed on the entire constraint set $\Xc_i$, which can be
prohibitively expensive when $\Xc_i$ is itself an intersection of many sets.
In addition, unlike the method in~\cite{NO2010},
DRP can also handle the cases when the projection on the entire set $\Xc_i$
is not possible since the set $\Xc_i$ may not be known in advance.

The challenges in convergence analysis of the DRP algorithm are posed mainly by its distributed nature,
through the {\it effects of the time-varying network}, and by the {\it projection errors} associated
with using projections on components
$\Xc_i^j$, $j\in I_i$ of the set $\Xc_i=\cap_{j\in I_i} \Xc_i^j$ instead of the projection on the set $\Xc_i$.
The fact that the DRP relies on a random component $\Xc_i^j$ poses particular difficulties,
as one needs to characterize the impact of the random projection errors, which is closely
related to errors in "set-approximations". To handle these difficulties, we make several mild assumptions.
We make an assumption on the random set processes $\{\Omega_i(k)\}$, $i\in V$,
that allows us to characterize the projection errors.
For the network we assume that it is sufficiently connected in order to properly conduct the information
among the agents. Finally, we assume that the agent weights are also properly chosen to ensure
that each agent is equally influencing every other agent. These network assumptions have
been typically used in distributed optimization algorithms over a time-varying network
(see e.g.~\cite{Tsiphd,Tsi1986,Bertsekas:1997,ANquan2008,AN2009,Ram2010}).
In the next subsections, we state our assumptions on the random set processes $\{\Omega_i(k)\}$, $i\in V$,
the network and the weight matrices $W(k)$.

\subsection{Assumptions on Random Set Process}
For the random sequences $\{\Omega_i(k)\}$, $i\in V$, we assume the following.
\begin{assumption}\label{assume:iid}
The sequences $\{\Omega_i(k)\}$, $i\in V$,
are \textit{iid} and independent of the initial random points $x_i(0)$, $i\in V$.
We have $\pi_i^j \triangleq \mathsf{Pr}\{\Omega_i(k) = j\} >0$ for all $j \in I_i$ and $i \in V$.
\end{assumption}
The variable $\Omega_i(k)$ can be viewed as a random sample at time $k$ of a random variable $\Omega_i$ that takes values $j \in I_i$ with probability $\pi_i^j$.
In some situations the probability distributions
$\pi_i$ may be dictated by nature and agent $i$ cannot control them.
In situations where the agents have all sets
$\Xc_i^j$, $j\in I_i$ available, each agent $i$ can choose a uniform distribution $\pi_i$ over the set $I_i$.

The next assumption is crucial in our analysis.
\begin{assumption}\label{assume:c}
For all $i\in V$, there exists a constant $c >0$ such that for all $x \in \mathbb{R}^d$,
\begin{equation}\label{eqn:c}
\mathrm{dist}^2(x,\mathcal{X}) \leq c \mathsf{E}\left[\mathrm{dist}^2(x,\mathcal{X}_i^{\Omega_i(k)})\right].
\end{equation}
\end{assumption}

Assumption \ref{assume:c} is satisfied, for example, when each set $\mathcal{X}_i^j$ is given by either linear inequality or a linear equality, or when the intersection set $\mathcal{X}$ has a nonempty interior.
In the first case, one can verify that the assumption holds  by using the results of Burke and Ferris on a set of weak sharp minima \cite{Burke93weaksharp}.
In the second case, one can use the ideas of the convergence rate analysis for the alternating projection algorithm of Gubin, Polyak and Raik in~\cite{Gubin19671}. In either case, the constant $c$ depends
on the probability distributions $\pi_i$ and some geometric properties of the sets.

\subsection{Assumptions on the Network and Weight Matrices}
We rely on the graphs $(V,E(k))$, $k\ge0$ to represent the time-varying network.
We make two assumptions.
\begin{assumption}\label{assume:nc}[Network Connectivity]
There exists a scalar $Q$ such that the graph
$\left(V,\bigcup_{\ell = 0,\ldots,Q-1} E(k+\ell)\right)$
is strongly connected for all $k\ge0$.
\end{assumption}
Assumption~\ref{assume:nc} ensures that the agents communicate
sufficiently often so that all functions and all constraints ($f_i$'s and $\mathcal{X}_i$'s)
influence the iterates of all agents.

Next, we make the following assumption on the edge weights (defined below~\eqref{eqn:algo3mix}).
\begin{assumption}\label{assume:ds}[Doubly Stochasticity]
For all $k \ge 0$,
\begin{enumerate}
\item[(a)] $[W(k)]_{ij} \ge 0$ and $[W(k)]_{ij}=0$ when $j \not\in N_i(k)$,
\item[(b)] $\sum_{j=1}^m W[(k)]_{ij} = 1$ for all $i \in V$,
\item[(c)] There exists a scalar $\eta \in (0,1)$ such that $[W(k)]_{ij} \ge \eta$ when $j \in N_i(k)$,
\item[(d)] $\sum_{i=1}^m [W(k)]_{ij} = 1$ for all $j \in V$.
\end{enumerate}
\end{assumption}
Assumption \ref{assume:ds}(a) states that the weights respect the network topology at any time $k$.
Assumption \ref{assume:ds}(b) means that each agent calculates
a weighted average of the estimates obtained from its neighbors.
Assumption \ref{assume:ds}(c) ensures that each agent gives sufficient weights on the information received.
Assumption \ref{assume:ds}(d) together with Assumption \ref{assume:nc} ensure that each agent is equally influential in the long run so that the agents arrive at a consensus on an optimal solution.

\section{Preliminaries  \label{sec:prelim}}
In this section, we state some definitions and results from the literature,
which will be used in later sections.

\noindent\textit{Convexity of Euclidean norm and its square.}
Both the Euclidean norm and its square are convex functions, i.e.,
for any vectors $v_1,\ldots,v_m \in \mathbb{R}^d$ and
nonnegative scalars $\beta_1,\ldots,\beta_m$ such that
$\sum_{i=1}^m \beta_i = 1$, we have
\begin{equation}\label{eqn:norm}
\left\|\sum_{i=1}^m \beta_i v_i \right\| \leq \sum_{i=1}^m \beta_i \|v_i\|,~
\left\|\sum_{i=1}^m \beta_i v_i \right\|^2 \leq \sum_{i=1}^m \beta_i \|v_i\|^2.
\end{equation}

\noindent\textit{Non-expansive projection property.}
We state a projection theorem (see \cite{BNO} for its proof).
\begin{lemma}\label{lem:proj}
Let $\mathcal{X} \subseteq \mathbb{R}^d$ be a nonempty closed convex set.
The function $\mathsf{\Pi}_{\mathcal{X}}: \mathbb{R}^d \rightarrow \mathcal{X}$ is continuous and nonexpansive, i.e.,
\begin{enumerate}
\item[(a)]
$\|\mathsf{\Pi}_{\mathcal{X}}[x]-\mathsf{\Pi}_{\mathcal{X}}[y]\| \leq \|x-y\|
\quad \hbox{for all }x,y \in \mathbb{R}^d. $
\item[(b)]
$\|\mathsf{\Pi}_{\mathcal{X}}[x]-y\|^2 \leq \|x-y\|^2 - \|\mathsf{\Pi}_{\mathcal{X}}[x]-x\|^2$ for all
$x \in \mathbb{R}^d$ and for all  $y \in \mathcal{X}.$
\end{enumerate}
\end{lemma}

\noindent\textit{Matrix convergence.}
Recall we defined $W(k)$ to be the matrix with $(i,j)$th entry equal to $w_{ij}(k)$.
From Assumption~\ref{assume:ds}, the matrix $W(k)$ is doubly stochastic.
Define for all $k, s$ with $k > s \ge 0$,
\begin{equation}\label{eqn:Phi}
\Phi(k,s) = W(k)W(k-1)\cdots W(s+1)W(s),
\end{equation}
with $\Phi(k,k)=W(k)$ for all $k\ge0$.
We state the convergence property of the matrix $\Phi(k,s)$ (see \cite{ANquan2008} for its proof).
Let $[\Phi(k,s)]_{ij}$ denote the $(i,j)$th entry of the matrix $\Phi(k,s)$, and $e \in \mathbb{R}^m$ be the column vector whose all entries are equal to 1.
\begin{lemma}\label{lem:Phi}
Let Assumptions \ref{assume:nc} and \ref{assume:ds} hold. Then,
\begin{enumerate}
\item[(a)] $\lim_{k \rightarrow \infty} \Phi(k,s) = \frac{1}{m}ee^T$ for all $s \ge 0$.
\item[(b)] $\left|[\Phi(k,s)]_{ij}- \frac{1}{m}\right| \le \theta \beta^{k-s}$ for all $k \ge s \ge 0$, where $\theta = \left(1-\frac{\eta}{4m^2}\right)^{-2}$ and $\beta = \left(1-\frac{\eta}{4m^2}\right)^{\frac{1}{Q}}$.
\end{enumerate}
\end{lemma}

\noindent\textit{Supermartingale convergence result.}
In our analysis of the DRP algorithm, we also make use of the following supermartingale convergence result due to Robbins and Siegmund (see~\cite[Lemma 10-11, p. 49-50]{polyak}).

\begin{theorem}\label{thm:super}
Let $\{v_k\}$, $\{u_k\}$, $\{a_k\}$ and $\{b_k\}$ be sequences of non-negative random variables such that
\[
\mathsf{E}[v_{k+1}|\Fc_k] \leq (1+a_k)v_k - u_k + b_k
\quad\text{ for all } k \geq 0 \quad a.s.,
\]
where $\Fc_k$ denotes the collection $v_0,\ldots,v_k$, $u_0,\ldots,u_k$, $a_0,\ldots,a_k$ and $b_0,\ldots,b_k$.
Also, let $\sum_{k=0}^{\infty} a_k < \infty$ and $\sum_{k=0}^{\infty} b_k < \infty$ a.s. Then, we have $\lim_{k \rightarrow \infty} v_k = v$ for a random variable $v \geq 0$ \textit{a.s.}, and $\sum_{k=0}^{\infty} u_k < \infty$ \textit{a.s}.
\end{theorem}

The above theorem is the key in our convergence analysis. Specifically,
once we show that Theorem~\ref{thm:super} applies to $v_{k+1}=
\sum_{i=1}^m\|x_i(k+1)-x^*\|^2$ for an optimal solution $x^*$,
the rest of the proof just builds on the implications
of the theorem.

\noindent\textit{Scalar Sequences.}
We also use the convergence result for scalar sequences (see Lemma~3.1 in~\cite{Ram2010} for its proof).
For a scalar $\beta$ and a scalar sequence $\{\gamma(k)\}$, we consider the convolution sequence $\sum_{\ell=0}^k \beta^{k-\ell}\gamma(\ell)$.
\begin{lemma}\label{lem:scalar}
If $\lim_{k \rightarrow \infty} \gamma(k) = \gamma$ and $0<\beta<1$, then
$
\lim_{k \rightarrow \infty} \sum_{\ell=0}^k\beta^{k-\ell}\gamma(\ell) = \frac{\gamma}{1-\beta}.
$
\end{lemma}

\section{Basic Relations \label{sec:lem}}
Our convergence analysis is based on a critical relation
that captures the decrease in values  $\sum_{i=1}^m\|x_i(k+1)-x^*\|^2$ as the algorithm progresses.
Such a relation is provided in Lemma~\ref{lem:first}, which is taken from~\cite{AN2011}
where it was developed for a centralized algorithm.
This basic relation  is further refined to take into account the distributed nature of the algorithm.
Specifically, in Lemma~\ref{lemma:key},
we show that the weighted averages $v_i(k)$ of the iterates approach the constraint set $X$ asymptotically.
Then, in Lemma~\ref{lem:disagree}, we prove that
the agents' disagreement on $v_i(k)$ is diminishing with the number $k$ of iterations.
The proof of Lemma~\ref{lem:disagree} relies on an
auxiliary result taken from~\cite{Ram2010},
which is provided in Lemma~\ref{lemma:ram}.

In the analysis, we will rely on the expectation taken with respect to the past history of the algorithm, which we define as follows.
Let $\mathcal{F}_k$ be the $\sigma$-algebra generated by the entire history of the algorithm up to time $k-1$
inclusively (realizations of all the random variables but not the realizations of the indices $\Omega_i$
at time $k$), i.e., for all $k \ge 1$,
\[
\Fc_k = \{x_i(0), i \in V\} \cup \{\Omega_i(\ell); 0\le \ell \le k-1, i \in V\},
\]
where $\Fc_0 = \{x_i(0), i \in V\}$.
Therefore, given $\Fc_k$, the collection $x_i(0),\ldots,x_i(k)$ and $v_i(0),\ldots,v_i(k)$
generated by the algorithm (\ref{eqn:algo3})-(\ref{eqn:algo4}) is fully determined.

\subsection{Basic Iterate Relation}\label{sec:basic}
The following lemma is from the paper \cite[Lemma 1]{AN2011}, which provides relation among the iterate obtained after one step of the algorithm (\ref{eqn:algo3}), a point in the feasible set $\Xc$ and an arbitrary point in $\mathbb{R}^d$.

\begin{lemma}\label{lem:first}
Let $\mathcal{Y} \subseteq \mathbb{R}^d$ be a closed convex set.
Let function $\phi:\mathbb{R}^d\to\mathbb{R}$ be convex and
differentiable over  $\mathbb{R}^d$ with Lipschitz continuous gradients with a constant $L$.
Let $y$ be given by
\[
y = \mathsf{\Pi}_{\mathcal{Y}}[x-\alpha\nabla \phi(x)] \quad
\text{for some } x \in \mathbb{R}^d \text{ and } \alpha > 0.
\]
Then, we have for any $\check{x} \in \mathcal{Y}$ and $z \in \mathbb{R}^d$,
\begin{align}\label{eqn:basic}
\| y-\check{x}\|^2 \leq {}& (1+A_{\tau}\alpha^2)\|x-\check{x}\|^2
  - 2\alpha(\phi(z)-\phi(\check{x})) \nonumber\\
{}& - \frac{3}{4}\|y-x\|^2 + \left(\frac{3}{8\tau}+2\alpha L\right)\|x-z\|^2 \nonumber\\
{}& + B_{\tau}\alpha^2\|\nabla \phi(\check{x})\|^2,
\end{align}
where $A_{\tau} = 8L^2 + 16\tau L^2$, $B_{\tau} = 8\tau+8$ and $\tau >0 $ is arbitrary.
\end{lemma}

Lemma~\ref{lem:first} provides a measure of progress toward an optimal point of the function $\phi$
when moving from a point $x$ in the direction opposite of the gradient $\nabla \phi(x)$.
Specifically, if $x^*$ is a minimizer of $\phi(x)$ over $\mathcal{Y}$,
the lemma (with $\check{x}=x^*$) will provide us with a relation between the distances
$\|y-x^*\|$ and $\|x-x^*\|$, where the point $y$ is
resulting from a projected-gradient step away from the point $x$.
The lemma provides a relation that helps us measure the progress
of a gradient-based algorithm for minimizing $\phi$.
Lemma~\ref{lem:first}, with a specific identification of the terms,
will be a starting point for our convergence proof.

\subsection{Projection Estimate \label{sec:proj}}
In the next lemma, we show that
the sequences $\{v_i(k)\}$, $i\in V$, approach the constraint set $\Xc$.
The result does not say that these sequences necessarily have accumulation points in $\Xc$,
but rather that the distance between $v_i(k)$ and the set $\Xc$ tends to 0, as $k\to\infty$, for all $i$.
Furthermore, these distances converge to 0 rather fast, as the sum of all squared distances over time is finite,
which is a critical relation in our analysis.

\begin{lemma}\label{lemma:key}
Let Assumption 1 hold.
Let each $W(k)$ be doubly stochastic, and let $\sum_{k=0}^\infty \a_k^2<\infty$. Then,
\[\sum_{k=0}^\infty \dist^2(v_i(k),\Xc)<\infty\qquad\hbox{for all $i\in V$ } a.s.\]
\end{lemma}
\begin{proof}
In Lemma~\ref{lem:first}, let $y=x_i(k+1)$, $x=v_i(k)$, $\mathcal{Y}=\mathcal{X}_i^{\Omega_i(k)}$,
$\alpha=\alpha_k$, $\phi=f_i$
and $\tau=c$ where $c$ is the constant from Assumption 3. Then, for any $\check{x} \in \Xc$ (also in
$\Xc_i^{\Omega_i(k)}$, since $\Xc\subseteq \Xc_i^{\Omega_i(k)}$)
and any $z\in\mathbb{R}^d$, we obtain
\allowdisplaybreaks{
\begin{align*}
\|x_i{}&(k+1)  -\check{x}\|^2
\leq  (1+A\alpha_k^2)\|v_i(k) - \check{x}\|^2\\
{}& - 2\a_k(f_i(z)-f_i(\check{x}))- \frac{3}{4}\|x_i(k+1)-v_i(k)\|^2\\
{}& + \left(\frac{3}{8c}+2\a_k L\right)\|v_i(k)-z\|^2 + B\a_k^2 G_f^2.
\end{align*}
}
where $A= 8L^2 + 16c L^2$ and $B = 8c+8$.
Here, we have also used
Assumption~\ref{assume:f}(d), according to which the gradients $\nabla f_i(x)$
are bounded on the set $\Xc$, i.e.,
$\|\nabla f_i(\mathsf{\Pi}_{\mathcal{X}}[v_i(k)])\|\le G_f$ for all $k$ and $i$.

Letting $\check{x}=z= \mathsf{\Pi}_{\mathcal{X}}[v_i(k)]$ in the preceding relation, we find
\begin{align}\label{eqn:e01}
\|x_i(k{}&+1)-\mathsf{\Pi}_{\mathcal{X}}[v_i(k)]\|^2
\leq  (1+A\alpha_k^2) \dist^2(v_i(k),\Xc) \nonumber\\
{}& - \frac{3}{4}\|x_i(k+1)-v_i(k)\|^2 \nonumber\\
{}& + \left(\frac{3}{8c}+2\a_k L\right)\dist^2(v_i(k),\Xc)
+ B\a_k^2 G_f^2.
\end{align}
By the definition of the projection, we have
\begin{align*}
\dist(x_i(k+1),\Xc)
={}& \|x_i(k+1)-\mathsf{\Pi}_{\mathcal{X}}[x_i(k+1)]\|\\
\le{}& \|x_i(k+1)-\mathsf{\Pi}_{\mathcal{X}}[v_i(k)]\|,
\end{align*}
\begin{align*}
\|x_i(k+1)-v_i(k)\|\ge{}& \left\|\mathsf{\Pi}_{\mathcal{X}_i^{\Omega_i(k) }}[v_i(k)]-v_i(k)\right\|\\
={}&\dist(v_i(k),\mathcal{X}_i^{\Omega_i(k)}).
\end{align*}
Upon substituting these estimates in~\eqref{eqn:e01}, we obtain
\begin{align}\label{eqn:e02}
\dist^2{}&(x_i(k+1),\Xc)
\leq  (1+A\alpha_k^2) \dist^2(v_i(k),\Xc) \nonumber\\
{}&- \frac{3}{4}\dist^2(v_i(k),\mathcal{X}_i^{\Omega_i(k)})\nonumber\\
{}& + \left(\frac{3}{8c}+2\a_k L\right)\dist^2(v_i(k),\Xc)
+ B\a_k^2 G_f^2. 
\end{align}
Taking the expectation in~\eqref{eqn:e02} conditioned on $\Fc_k$,
and using
\[\EXP{\dist^2(v_i(k),\mathcal{X}_i^{\Omega_i(k)})\mid \Fc_k}\ge
\frac{1}{c}\dist^2(v_i(k),\Xc),\]
which follows by Assumption 3,
we find that almost surely
\begin{align}\label{eqn:e03}
 \mathsf{E}\Big[{}&\dist^2(x_i(k+1),\Xc)\mid\Fc_k\Big] \leq (1\hspace{-1mm}+\hspace{-1mm}A\alpha_k^2) \dist^2(v_i(k),\Xc)\nonumber\\
{}& - \left(\frac{3}{8c}-2\a_k L\right)\dist^2(v_i(k),\Xc)
 + B\a_k^2 G_f^2. 
\end{align}

By using the definition of $v_i(k)$ (as a convex combination of $x_j(k)$ in~\eqref{eqn:algo3mix})
and the convexity
of the distance function $x\mapsto\dist^2(x,\Xc)$ (see \cite[p.~88]{BNO}),
we find that
\[\dist^2(v_i(k),\Xc)\le \sum_{j=1}^m [W(k)]_{ij}\,\dist^2(x_j(k),\Xc).\]
The preceding relation and~\eqref{eqn:e03} imply that almost surely for all $k\ge 0$,
\allowdisplaybreaks{
\begin{align*}
 \mathsf{E}\Big[{}&\dist^2(x_i(k+1),\Xc)\mid\Fc_k\Big]\nonumber\\
 \leq {}&(1+A\alpha_k^2) \sum_{j=1}^m [W(k)]_{ij}\,\dist^2(x_j(k),\Xc) \nonumber\\
{}& - \left(\frac{3}{8c}-2\a_k L\right)\dist^2(v_i(k),\Xc)
+ B\a_k^2G_f^2.
\end{align*}
}
Finally, by summing over all $i$ and using the fact that each $W(k)$ has column sums equal to 1,
we arrive at the following relation: almost surely for all $k\ge 0$,
\begin{align*}
\mathsf{E}\Big[ &\sum_{i=1}^m{}\dist^2(x_i(k+1),\Xc)\mid\Fc_k\Big]\nonumber\\
\leq {}& (1+A\alpha_k^2) \sum_{j=1}^m \dist^2(x_j(k),\Xc) \nonumber\\
{}& - \left(\frac{3}{8c}-2\a_k L\right)\sum_{i=1}^m\dist^2(v_i(k),\Xc)
+ mB\a_k^2G_f^2.
\end{align*}
Since $\sum_{k=0}^\infty \a_k^2<\infty$, it follows that $\a_k \rightarrow 0$,
implying that there exists $\bar{k}$ such that
$\frac{3}{8c}-2\a_k L >0 $ for all $k \ge \bar{k}$.
Therefore, for all $k \ge \bar{k}$, all the conditions of the supermartingale theorem are satisfied
(Theorem~\ref{thm:super}).
By applying the supermartingale theorem (to a time-delayed process from $\bar k$ onward) we conclude that
\[\sum_{k=0}^\infty \dist^2(v_i(k),\Xc)<\infty\qquad\hbox{for all $i \in V$ }  a.s.\]
\end{proof}

Lemma~\ref{lemma:key} shows that the points $v_i(k)$ are getting close to the set $\Xc$ relatively fast,
as $k\to\infty$. If the set $\Xc$ was compact, this would imply that all accumulation points
of $\{v_i(k)\}$ would lie in the set $\Xc$. However, there would be no guarantee that
the accumulation points of any two sequences $\{v_i(k)\}$ and $\{v_j(k)\}$ would be the same.
Even worse,
Lemma~\ref{lemma:key} would give no information about optimality of any of the accumulation points.
In the next section, we provide a result that helps us claim later on that
any two sequences $\{v_i(k)\}$ and $\{v_j(k)\}$ have the same accumulation points.

\subsection{Disagreement Estimate}
We now quantify the agent disagreements in time. We measure
the disagreements by using the norm $\|v_i(k)-\bar{v}(k)\|$ of the differences between
the estimates $v_i(k)$ generated by different agents according the algorithm
(\ref{eqn:algo3})-(\ref{eqn:algo4}) and their instantaneous average
$\bar{v}(k) = \frac{1}{m}\sum_{\ell=1}^m v_\ell(k)$. The proof of our result relies on
a lemma (adopted from \cite[Theorem 4.2]{Ram2012}), which states that
the iterates generated by a ``perturbed" consensus protocol are guaranteed to arrive at a consensus
when the perturbations are small in some sense. This lemma is provided next.
\begin{lemma}\label{lemma:ram}
Let Assumptions \ref{assume:nc} and \ref{assume:ds} hold.
Consider the iterates generated by
\begin{equation}\label{eqn:rule}
\theta_i(k\hspace{-0.5mm}+\hspace{-0.5mm}1) \hspace{-0.5mm}= \hspace{-0.5mm}\sum_{j=1}^m [W(k)]_{ij} \theta_j(k) + \e_i(k) \text{ for all } i \in V.
\end{equation}
Suppose there exists a non-negative non-increasing scalar sequence $\{\alpha_k\}$ such that
$\sum_{k=0}^{\infty} \alpha_k\|\e_i(k)\| < \infty $ for all $i \in V.$
Then, for all $i,j \in V$,
\[
\sum_{k=0}^{\infty} \alpha_k \|\theta_i(k) -\theta_j(k)\| < \infty.
\]
\end{lemma}

Using Lemma \ref{lemma:ram},
we prove the following disagreement results that will be important in our analysis later.

\begin{lemma} \label{lem:disagree}
Let Assumptions \ref{assume:f},
\ref{assume:nc} and \ref{assume:ds} hold. Also, assume that the stepsize sequence
$\{\a_k\}$ is non-increasing and such that $\sum_{k=0}^\infty \a_k^2<\infty$.
Define
\[
e_i(k)=x_i(k+1)-v_i(k)   \qquad\hbox{for all $i\in V$ and $k\ge0$}.
\]
Then, we have almost surely
\begin{equation}\label{eqn:lem_e}
\sum_{k=0}^\infty \|e_i(k)\|^2 <\infty    \quad \hbox{for all $i \in V$},
\end{equation}
\begin{equation}\label{eqn:lem_disagree}
\sum_{k=0}^\infty\a_k\|v_i(k)-\bar v(k)\|<\infty    \quad \hbox{for all $i \in V$},
\end{equation}
where $\bar v(k)=\frac{1}{m}\sum_{\ell=1}^m v_\ell(k)$.
\end{lemma}

\def\provik{\mathsf{\Pi}_{\mathcal{X}}[v_i(k)]}
\begin{proof}
Define $z_i(k) \triangleq \mathsf{\Pi}_{\mathcal{X}}[v_i(k)]$.
Consider $\|e_i(k)\|$, for which we can write
\begin{align*}
\|e_i(k)\|
\le{}& \|x_i(k+1) - z_i(k)\| +\|z_i(k) - v_i(k)\|\\
= {}&\left\|\mathsf{\Pi}_{\mathcal{X}_i^{\Omega_i(k)}}[v_i(k) - \alpha_k \nabla f_i(v_i(k))]
- z_i(k)\right\|\\
{}&+\left\|z_i(k)- v_i(k)\right\|.
\end{align*}
Since $\mathcal{X} \subseteq \mathcal{X}_i^{\Omega_i(k)}$ and
$z_i(k) \in \mathcal{X}$, we have $z_i(k)\in \mathcal{X}_i^{\Omega_i(k)}$.
Using the projection theorem (Lemma \ref{lem:proj}), we obtain
\allowdisplaybreaks{
\begin{align}\label{eqn:e_final}
\|e_i&(k)\| \nonumber\\
\leq & \|v_i(k) - \alpha_k \nabla f_i(v_i(k)) -z_i(k)\|+\|z_i(k)- v_i(k)\|\nonumber\\
\leq & 2\| v_i(k)- z_i(k)\| + \alpha_k \|\nabla f_i(v_i(k))\|\nonumber\\
\le & 2\|v_i(k) - z_i(k)\| +\alpha_k\|\nabla f_i(z_i(k))\| \nonumber\\
&+ \alpha_k \|\nabla f_i(v_i(k))-\nabla f_i(z_i(k))\|\nonumber\\
\le & (2+\a_0L)\| v_i(k) - z_i(k)\| + \alpha_k G_f,
\end{align}
}
where the last inequality follows by using $\a_k\le \a_0$, the Lipschitz gradient property of $f_i$
and the gradient boundedness property (Assumptions~\ref{assume:f}(c) and \ref{assume:f}(d)).
Therefore, applying $(a+b)^2 \le 2a^2 + 2b^2$ in inequality (\ref{eqn:e_final}), we have for all $i\in V$ and $k\ge0$,
\begin{equation}\label{eqn:e}
\|e_i(k)\|^2 \leq 2(2+\a_0 L)^2\| v_i(k) - z_i(k)\|^2 + 2\alpha_k^2 G_f^2.
\end{equation}
Recall that we defined $z_i(k) \triangleq \mathsf{\Pi}_{\Xc}[v_i(k)]$,
so we have $\|v_i(k) - z_i(k)\| = \dist (v_i(k),\Xc)$.
In the light of Lemma~\ref{lemma:key}, we also have $\sum_{k=0}^\infty \|v_i(k) - z_i(k)\|^2 <\infty$
almost surely.
Since $\sum_{k=0}^\infty\a_k^2<\infty$, we conclude that
\[
\sum_{k=0}^\infty \|e_i(k)\|^2 <\infty \quad \text{for all } i \in V \textit{ a.s.}
\]

By applying the inequality $2ab\le a^2+b^2$ to each term $\a_k\|e_i(k)\|$, we see that for all $i \in V$ almost surely
\begin{align*}
\sum_{k=0}^\infty \a_k\|e_i(k)\|\le
\frac{1}{2}\sum_{k=0}^\infty\a_k^2 + {}&\frac{1}{2}\sum_{k=0}^\infty\|e_i(k)\|^2<\infty.
\end{align*}

Now, we note that $x_i(k+1)=v_i(k)+e_i(k)$
with $v_i(k)=\sum_{j=1}^m[W(k)]_{ij} x_j(k)$
and  the error $e_i(k)$ satisfying $\sum_{k=0}^\infty \a_k\|e_i(k)\|<\infty$ almost surely.
Therefore, by Lemma~\ref{lemma:ram}, it follows that
\begin{equation}\label{eqn:sumfin}
\sum_{k=0}^\infty\a_k\|x_i(k)-x_j(k)\|<\infty \hbox{ for all $i$ and $j$ }  a.s.
\end{equation}
Next, we consider $\|v_i(k)- \bar v(k)\|$.
Recalling that $v_i(k)=\sum_{j=1}^m [W(k)]_{ij}\, x_j(k)$ (see~\eqref{eqn:algo3mix})
and $W(k)$ is stochastic (Assumption~\ref{assume:ds}), and by using the
convexity of the norm, we obtain
\begin{align*}
\|v_i(k)- \bar v(k)\|\le {}& \sum_{j=1}^m w_{ij}(k)\left\|x_j(k)-\bar v(k)\right\|\\
\le {}&\sum_{j=1}^m \left\|x_j(k)- \frac{1}{m}\sum_{\ell=1}^m x_\ell(k)\right\|,
\end{align*}
where in the last equality we use $0\le [W(k)]_{ij}\le 1$ and
$\bar v(k)=\frac{1}{m}\sum_{\ell=1}^m x_\ell(k)$,
which  holds since $v_i(k)=\sum_{j=1}^m [W(k)]_{ij}\, x_j(k)$ and each $W(k)$ is doubly stochastic.
Therefore, by using the convexity of the norm again, we see
\[
\sum_{j=1}^m\hspace{-0.5mm} \left\|x_j(k) \hspace{-0.5mm}-\hspace{-0.5mm} \frac{1}{m}\hspace{-0.5mm}\sum_{\ell=1}^m x_\ell (k)\right\|
\hspace{-0.5mm}\le\hspace{-0.5mm} \frac{1}{m}\hspace{-0.5mm}\sum_{j=1}^m \sum_{\ell=1}^m\hspace{-0.5mm}\left\|x_j(k)\hspace{-0.5mm}-\hspace{-0.5mm} x_\ell(k)\right\|.
\]
We thus have
\[\a_k\|v_i(k)- \bar v(k)\|\le \frac{\a_k}{m}\sum_{j=1}^m \sum_{\ell=1}^m\left\|x_j(k)- x_\ell(k)\right\|,\]
and by using the relation in~\eqref{eqn:sumfin}, we conclude that
\[\sum_{k=0}^\infty\a_k\|v_i(k)- \bar v(k)\|<\infty    \qquad \hbox{for all $i\in V$ } a.s.\]
\end{proof}

\section{Almost Sure Convergence of DRP Algorithm\label{sec:conv}}
We are now ready to assert the convergence of the method (\ref{eqn:algo3})-(\ref{eqn:algo4})
using the lemmas established in Section \ref{sec:lem}.
To outline the rough idea of the proof, let us note that
Lemma~\ref{lemma:key} allows us to infer that $v_i(k)$ approaches the set $\Xc$.
Lemma~\ref{lem:disagree} will allow us to claim that any two sequences $\{v_i(k)\}$ and $\{v_j(k)\}$
have the same accumulation almost surely, under some mild assumptions on the stepsize.
To claim the convergence of the iterates to an optimal solution, it remains to relate
the accumulation points of $\{v_i(k)\}$ to the optimal solutions of problem~\eqref{eqn:prob}.
This last piece is provided by the iterate relation of Lemma~\ref{lem:first}, supported
by the supermartingale theorem.

From here onward, we use the following notation regarding the optimal value and optimal
solutions of problem~\eqref{eqn:prob}:
\[
f^* = \min_{x\in\mathcal{X}} f(x), \qquad \mathcal{X}^* = \{x \in \mathcal{X} \mid f(x) = f^*\}.
\]

We have the following convergence result.

\begin{proposition} \label{prop:as}
Let Assumptions \ref{assume:f}-\ref{assume:ds} hold.
Let the stepsize be such that $\sum_{k=0}^{\infty} \alpha_k = \infty$
and $\sum_{k=0}^{\infty} \alpha_k^2 < \infty$.
Assume that problem~\eqref{eqn:prob} has a nonempty optimal set $\Xc^*$.
Then, the iterates $\{x_i(k)\}$, $i\in V$, generated by
the method~(\ref{eqn:algo3})-(\ref{eqn:algo4}) converge almost surely
to some random point in the optimal set $\Xc^*$, i.e.,
for some random vector $x^\star\in \Xc^*$,
\[
\lim_{k \rightarrow \infty} x_i(k) = x^\star \qquad \text{for all}\quad i \in V \ a.s.
\]
\end{proposition}
\begin{proof}
We use the definition of the iterate $x_i(k)$ in (\ref{eqn:algo3})-(\ref{eqn:algo4})
and lemma \ref{lem:first} with the following identification: $\mathcal{Y} = \mathcal{X}_i^{\Omega_i(k)}$,
$y = x_i(k+1)$, $x = v_i(k)$, $z = z_i(k) \triangleq \mathsf{\Pi}_{\mathcal{X}}[v_i(k)]$, $\alpha = \alpha_k$ and $\tau = c$ where $c$ is the constant from the relation (\ref{eqn:c}).
Thus, for any $\check{x} \in \mathcal{X}$, $k \geq 0$ and $i \in V$, we have
\begin{align*}
\|{}&x_i(k+1)-\check{x}\|^2 \leq  (1+A\alpha_k^2) \|v_i(k)-\check{x}\|^2\nonumber\\
{}& - 2\alpha_k (f_i(z_i(k))-f_i(\check{x})) - \frac{3}{4}\|x_i(k+1)-v_i(k)\|^2 \nonumber\\
{}& + \left(\frac{3}{8c}+2\alpha_kL\right)\|v_i(k)-z_i(k)\|^2 + B\alpha_k^2\|\nabla f_i(\check{x})\|^2,
\end{align*}
with $A = 8L^2+16cL^2$ and $B = 8c+8$.
We next sum the preceding relations over $i=1,\ldots,m$.
Also, we use the convexity of the squared-norm (cf.~\eqref{eqn:norm})
and the doubly stochasticity of the weights to obtain the following relation:
\allowdisplaybreaks{
\begin{align*}
\sum_{i=1}^m \|v_i(k)-\check{x}\|^2
\leq {}&\sum_{i=1}^m\sum_{j=1}^m [W(k)]_{ij}\|x_j(k)-\check{x}\|^2\nonumber\\
= {}&\sum_{j=1}^m \left(\sum_{i=1}^m [W(k)]_{ij}\right)\|x_j(k)-\check{x}\|^2\nonumber\\
= {}&\sum_{j=1}^m \|x_j(k)-\check{x}\|^2.
\end{align*}
}
By doing so, and taking into account that the gradients $\|\nabla f_i(\check{x})\|$ are bounded over $\Xc$
by a scalar $G_f$ (Assumption~\ref{assume:f}(d)),
we obtain for any $\check{x} \in \Xc$ and $k\ge 0$,
\begin{align}\label{eqn:eq}
{}&\sum_{i=1}^m \|x_i(k+1)-\check{x}\|^2 \leq
 (1+A\alpha_k^2) \sum_{i=1}^m \|x_i(k)-\check{x}\|^2\nonumber\\
{}&-\hspace{-0.5mm} 2\alpha_k \hspace{-0.5mm}\sum_{i=1}^m \hspace{-0.5mm}(f_i(z_i(k))\hspace{-1mm}-\hspace{-1mm}f_i(\check{x})\hspace{-0.5mm})-\hspace{-0.5mm} \frac{3}{4}\hspace{-0.5mm}\sum_{i=1}^m\hspace{-0.5mm} \|x_i(k\hspace{-0.5mm}+\hspace{-0.5mm}1)\hspace{-1mm}-\hspace{-1mm}v_i(k)\|^2 \nonumber\\
{}& + \hspace{-0.5mm}\left(\hspace{-0.5mm}\frac{3}{8c}\hspace{-0.5mm}+\hspace{-0.5mm}2\alpha_kL\hspace{-0.5mm}\right)\hspace{-0.5mm}\sum_{i=1}^m\hspace{-0.5mm} \|v_i(k)\hspace{-0.5mm}-\hspace{-0.5mm}z_i(k)\|^2
\hspace{-0.5mm}+\hspace{-0.5mm} mB\alpha_k^2 G_f^2. 
\end{align}

Let $\bar{z}(k) \triangleq \frac{1}{m}\sum_{\ell =1}^m z_\ell(k)$ and recall that $f(x)=\sum_{i=1}^m f_i(x)$.
Using $\bar{z}(k)$ and $f$,
we can rewrite the second term on the right hand side in (\ref{eqn:eq}) as follows.
\begin{align}\label{eqn:rewrite}
\sum_{i=1}^m (f_i(z_i(k))-f_i(\check{x})) {}
 = & \sum_{i=1}^m (f_i(z_i(k))-f_i(\bar{z}(k))) \nonumber\\
 & +  (f(\bar{z}(k))-f(\check{x})).
\end{align}
We estimate the first term on the right hand side of the above equation as follows.
Using the convexity of each function $f_i$, we obtain
\begin{align*}
\sum_{i=1}^m & (f_i(z_i(k))\hspace{-0.5mm}-\hspace{-0.5mm}f_i(\bar{z}(k)))
\hspace{-0.5mm} \geq \hspace{-0.5mm}\sum_{i=1}^m  \langle \nabla f_i(\bar{z}(k),z_i(k)\hspace{-0.5mm}-\hspace{-0.5mm}\bar{z}(k)\rangle\nonumber\\
&\ge -\sum_{i=1}^m  \|\nabla f_i(\bar{z}(k))\|\,\|z_i(k)-\bar{z}(k)\|.
\end{align*}
Since $\bar{z}(k)$ is a convex combination of points $z_i(k)\in \Xc$, it follows that $\bar z(k) \in \Xc$.
This observation and Assumption~\ref{assume:f}(d), stating that the gradients $\nabla f_i(x)$ are uniformly bounded for $x\in\Xc$, yield
\begin{align}\label{eqn:a}
\sum_{i=1}^m\hspace{-0.5mm} (f_i(z_i(k))\hspace{-0.5mm}-\hspace{-0.5mm}f_i(\bar{z}(k)))
\hspace{-0.5mm}\ge\hspace{-0.5mm} -G_f\hspace{-1mm}\sum_{i=1}^m \hspace{-0.5mm}\|z_i(k)\hspace{-0.5mm}-\hspace{-0.5mm}\bar{z}(k)\|.
\end{align}
 We next consider the term $\|z_i(k)-\bar{z}(k)\|$, for which by using
 $\bar{z}(k) \triangleq \frac{1}{m}\sum_{\ell =1}^m z_\ell(k)$ we have
 \begin{align*}
 \|&z_i(k)-\bar{z}(k)\|=\left\|\frac{1}{m}\sum_{\ell=1}^m(z_i(k)-z_\ell(k))\right\|\nonumber\\
 \le& \frac{1}{m}\sum_{\ell=1}^m\|z_i(k)-z_\ell(k)\|
 \le \frac{1}{m}\sum_{\ell=1}^m\|v_i(k)-v_\ell(k)\|,
 \end{align*}
 where the first inequality is obtained by the convexity of the norm (see \eqref{eqn:norm})
 and the last inequality follows by the non-expansive projection property (Lemma \ref{lem:proj}).
 Furthermore, by using $\|v_i(k)-v_\ell(k)\|\le \|v_i(k)-\bar v(k)\|+ \|v_\ell(k)- \bar v(k)\|$,
 we obtain for every $i\in V$,
 \[\|z_i(k)-\bar{z}(k)\|\le \|v_i(k)-\bar v(k)\|+\frac{1}{m}\sum_{\ell=1}^m\|v_\ell(k)-\bar v(k)\|.\]
Upon summing over $i\in V$, we find that
\begin{align}\label{eqn:sums}
\sum_{i=1}^m \|z_i(k)-\bar{z}(k)\|\le 2 \sum_{i=1}^m\|v_i(k)-\bar v(k)\|.
\end{align}
Combining relations~\eqref{eqn:sums} and~\eqref{eqn:a},
and substituting the resulting relation in
equation (\ref{eqn:rewrite}), we find that
\begin{align*}
\sum_{i=1}^m (f_i(z_i(k))-f_i(\check{x})) {}
\ge& -2G_f \sum_{i=1}^m \|v_i(k)-\bar{v}(k)\|\nonumber\\
&+  (f(\bar{z}(k))-f(\check{x})).
\end{align*}
Finally, by using the preceding estimate in inequality (\ref{eqn:eq}),
we obtain for any $\check{x} \in \mathcal{X}$ and $k \geq 0$,
\begin{align}\label{eqn:rewrite0}
{}&\sum_{i=1}^m \|x_i(k+1)-\check{x}\|^2 \leq  (1+A\alpha_k^2) \sum_{i=1}^m \|x_i(k)-\check{x}\|^2 \nonumber\\
{}& - 2\alpha_k (f(\bar{z}(k))-f(\check{x}))
  - \frac{3}{4}\sum_{i=1}^m \|x_i(k+1)-v_i(k)\|^2\nonumber\\
{}& + \left(\frac{3}{8c}+2\alpha_kL\right)\sum_{i=1}^m \|v_i(k)-z_i(k)\|^2 \nonumber\\
{}&  + 4\alpha_k G_f\sum_{i=1}^m \|v_i(k)-\bar{v}(k)\|
+ mB\alpha_k^2 G_f^2.
\end{align}

By the definition of $x_i(k+1)$, we have $x_i(k+1) \in \Xc_i^{\Omega_i(k)}$, which implies $\|x_i(k+1)-v_i(k)\|\ge \dist(v_i(k),\Xc_i^{\Omega_i(k)})$ for $i \in V$. Also, from the definition of $z_i(k) \triangleq \mathsf{\Pi}_{\Xc}[v_i(k)]$, we have $\|v_i(k)-z_i(k)\| = \dist(v_i(k),\Xc)$ for $i \in V$.
Using these relations and letting $\check{x}=x^*$ for an arbitrary $x^*\in \Xc^*$,  from~\eqref{eqn:rewrite0}
we obtain for all $k\ge 0$,
\allowdisplaybreaks{
\begin{align*}
\sum_{i=1}^m{}& \|x_i(k+1)- x^*\|^2 \leq  (1+A\alpha_k^2) \sum_{i=1}^m\|x_i(k)-x^*\|^2 \nonumber\\
{}& - 2\alpha_k(f(\bar{z}(k))-f^*)
 - \frac{3}{4}\sum_{i=1}^m\dist^2(v_i(k),\Xc_i^{\Omega_i(k)}) \nonumber\\
 {}& + \left(\frac{3}{8c}+2\alpha_kL\right)\sum_{i=1}^m\dist^2(v_i(k),\Xc) \nonumber\\
{}&  + 4 \alpha_k G_f \sum_{i=1}^m \|v_i(k)-\bar{v}(k)\|
+ mB\alpha_k^2 G_f^2.
\end{align*}
}
By taking the expectation conditioned on $\mathcal{F}_{k}$, and noting that $x_i(k)$, $v_i(k)$, $\bar{v}(k)$, and $\bar{z}(k)$ are fully determined by $\mathcal{F}_{k}$, we have almost surely
for all $x^*\in \Xc$ and $k\ge 0$,
\allowdisplaybreaks{
\begin{align*}
{}&\mathsf{E}\Big[\sum_{i=1}^m \|x_i(k+1)-x^*\|^2\mid\mathcal{F}_k\Big] \nonumber\\
 \leq {}& (1+A\alpha_k^2) \sum_{i=1}^m\|x_i(k)-x^*\|^2  - 2\alpha_k(f(\bar{z}(k))-f^*)  \nonumber\\
{} &  - \frac{3}{4}\mathsf{E}\left[\sum_{i=1}^m\dist^2(v_i(k),\Xc_i^{\Omega_i(k)})~|~\mathcal{F}_k\right] \nonumber\\
{}&+  \left(\frac{3}{8c}+2\alpha_kL\right)\sum_{i=1}^m\dist^2(v_i(k),\Xc) \nonumber\\
{}&  + 4\alpha_k G_f \sum_{i=1}^m\|v_i(k)-\bar{v}(k)\|
+ mB\alpha_k^2 G_f^2.
\end{align*}
}
By Assumption \ref{assume:c}, we have
$\mathrm{dist}^2(x,\mathcal{X}) \leq c \mathsf{E}\left[\mathrm{dist}^2(x,\mathcal{X}_i^{\Omega_i(k)})~|~\mathcal{F}_k\right] $
for all $x \in \Xc$ and all $i\in V$.
Furthermore, since $\alpha_k \rightarrow 0$, by choosing $\bar{k}$ large enough so that $2\alpha_kL\le \frac{3}{8c}$, we have for all $k \ge \bar{k}$,
\begin{align*}
- \frac{3}{4}{}&\mathsf{E}\left[\sum_{i=1}^m\dist^2(v_i(k),\Xc_i^{\Omega_i(k)})~|~\mathcal{F}_k\right] \\
{}& + \left(\frac{3}{8c}+2\alpha_kL\right)\sum_{i=1}^m\dist^2(v_i(k),\Xc) \le 0.
\end{align*}
Thus, we obtain almost surely for all $k \ge \bar{k}$ and $x^*\in \Xc^*$,
\begin{align}\label{eqn:final}
\mathsf{E}{}&\Bigg[\sum_{i=1}^m\|x_i(k+1)-x^*\|^2~|~\mathcal{F}_k\Bigg] \nonumber\\
\leq {}& (1+A\alpha_k^2) \sum_{i=1}^m\|x_i(k)-x^*\|^2  \nonumber\\
{}&- 2\alpha_k(f(\bar{z}(k))-f^*)  \nonumber\\
{}&+ 4\alpha_k G_f \sum_{i=1}^m\|v_i(k)-\bar{v}(k)\|
+ mB\alpha_k^2G_f^2.
\end{align}
Since $\bar{z}(k) \in \Xc$, we have $f(\bar{z}(k)) - f^* \ge 0$.
Thus, under the assumption $\sum_{k=0}^{\infty} \alpha_k^2 < \infty$ and Lemma \ref{lem:disagree}, relation~(\ref{eqn:final}) satisfies all the conditions of the supermartingale convergence of
Theorem~\ref{thm:super}.
Hence, the sequence $\{\|x_i(k)-x^*\|^2\}$ is convergent almost surely
for any $i \in V$ and $x^* \in \Xc^*$, and
\[
\sum_{k=0}^{\infty} \alpha_k(f(\bar{z}(k))-f(x^*)) < \infty \quad a.s.
\]
The preceding relation and the condition $\sum_{k=0}^{\infty} \alpha_k= \infty$ imply that
\begin{equation}\label{eqn:res1}
\liminf_{k\rightarrow \infty} (f(\bar{z}(k))-f(x^*))=0\quad a.s.
\end{equation}
By Lemma~\ref{lemma:key}, noting that $z_i(k)=\mathsf{\Pi}_{\Xc}[v_i(k)]$, we have
$\sum_{k=1}^{\infty}\sum_{i=1}^m\|v_i(k)-z_i(k)\|^2 < \infty$ almost surely,
implying
\begin{equation}\label{eqn:res2}
\lim_{k\rightarrow \infty}\|v_i(k)-z_i(k)\|=0 \quad \text{for all } i \in V \textit{ a.s.}
\end{equation}

Recall that
the sequence $\{\|x_i(k)-x^*\|\}$ is convergent almost surely for all $i \in V$ and every $x^* \in \Xc^*$.
Then, in view of relation (\ref{eqn:algo3}), we have that
the sequence $\{\|v_i(k)-x^*\|\}$ is also
convergent almost surely for all $i \in V$ and $x^* \in \Xc^*$.
By relation~(\ref{eqn:res2}) it follows that
$\{\|z_i(k)-x^*\|\}$ is also convergent almost surely for all $i \in V$ and $x^* \in \Xc^*$.
Since $\|\bar v_i(k)-x^*\|\le \frac{1}{m}\sum_{i=1}^m\|v_i(k)-x^*\|$
and  the sequence $\{\|v_i(k)-x^*\|\}$ is convergent almost surely for all $i \in V$ and $x^* \in \Xc^*$,
it follows that $\{\|\bar v(k)-x^*\|\}$ is convergent almost surely for all $x^* \in \Xc^*$.
Using a similar argument, we can conclude that
$\{\|\bar{z}(k)-x^*\|\}$ is convergent almost surely for all $x^* \in \Xc^*$.
As a particular consequence, it follows that the sequences $\{\bar v(k)\}$ and $\{\bar z(k)\}$
are almost surely bounded and, hence, they have accumulation points.
From relation (\ref{eqn:res1}) and the continuity of $f$, it follows that
the sequence $\{\bar{z}(k)\}$ must have one accumulation point in the set $\Xc^*$ almost surely.
This and the fact that $\{\|\bar{z}(k)-x^*\|\}$ is convergent almost surely for every $x^* \in \Xc^*$ imply that for a random point $x^\star \in \Xc^*$,
\begin{equation}\label{eqn:z_final}
\lim_{k \rightarrow \infty} \bar{z}(k) = x^\star \quad a.s.
\end{equation}
Now, from $\bar{z}(k) = \frac{1}{m}\sum_{\ell=1}^m z_\ell(k)$
and $\bar{v}(k) = \frac{1}{m}\sum_{i=\ell}^m v_\ell(k)$, using
relation~(\ref{eqn:res2}) and the convexity of the norm (cf.~\eqref{eqn:norm}), we obtain almost surely
\[\lim_{k\to\infty}\|\bar{v}(k)-\bar{z}(k)\|\le\frac{1}{m}\sum_{\ell=1}^m
\lim_{k\to\infty}\|v_\ell(k)-z_\ell(k)\|=0.\]
In view of relation~\eqref{eqn:z_final}, it follows that
\begin{equation}\label{eqn:v_final}
\lim_{k \rightarrow \infty} \bar{v}(k) = x^* \quad a.s.
\end{equation}
By relation (\ref{eqn:lem_disagree}) in Lemma \ref{lem:disagree}, we have
\begin{align}\label{eqn:liminfv}
\liminf_{k \rightarrow \infty} \|v_i(k)-\bar{v}(k)\| = 0 \quad \text{for all } i \in V \textit{ a.s.}
\end{align}
The fact that $\{\|v_i(k)-x^*\|\}$ is convergent almost surely for all $i$, together with~\eqref{eqn:v_final}
and~\eqref{eqn:liminfv}  implies that
\begin{equation}\label{eqn:consensus}
\lim_{k \rightarrow \infty} \|v_i(k)-x^\star\| = 0 \quad \text{for } i \in V \textit{ a.s.}
\end{equation}
Finally, from relation (\ref{eqn:lem_e}) in Lemma \ref{lem:disagree}, we have
$\lim_{k \rightarrow \infty} \|x_i(k+1)-v_i(k)\| = 0$ for all $i \in V$ almost surely,
which together with the limit in~\eqref{eqn:consensus} yields
$\lim_{k\rightarrow \infty} x_i(k) = x^\star$ for all $i \in V$ almost surely.

\end{proof}

\section{Distributed Mini-Batch Random Projection Algorithm}\label{sec:mini-batches}
As an extension of the algorithm in~\eqref{eqn:algo3}--\eqref{eqn:algo4}, one may
consider an algorithm where the agents use several random projections at each iteration.
Namely, after generating $v_i(k)$ each agent may take (or nature may reveal them)
several random samples $\Omega_i^1(k),\ldots,\Omega^b_i(k)$, where each $\Omega_i^r(k)\in I_i$ and
$b\ge 1$ is the batch-size.
Each collection $\Omega_i^1(k),\ldots,\Omega^b_i(k)$ consists of mutually independent random variables
and is independent of the past realizations.
More specifically, we have $b$ random independent samples of the \textit{iid} random variable
$\Omega_i(k)$ (taking values in $I_i$).
Using the compact form~\eqref{eqn:algo3mix} for the update in~\eqref{eqn:algo3}, in the mini-batch version
of the algorithm, each agent $i\in V$,
performs the following steps:
\allowdisplaybreaks{
\begin{subequations}
\begin{align}
v_i(k) = {}&  \sum_{j=1}^m [W(k)]_{ij}\, x_j(k), \label{eqn:algo3batch}\\
\psi^0_i(k) = {}& v_i(k)-\alpha_k\nabla f_i(v_i(k)),\label{eqn:initproj}\\
\psi^r_i(k) = {}& \mathsf{\Pi}_{\mathcal{X}_i^{\Omega^r_i(k)}} [\psi_i^{r-1}(k)]\hbox{ for }r=1,\ldots,b,
\label{eqn:sucproj}\\
x_i(k+1) = {}& \psi_i^b(k),
\label{eqn:algo4batch}
\end{align}
\end{subequations}
}
where $\alpha_k > 0$ is a stepsize at time $k$ and $x_i(0) \in \mathbb{R}^d$ is an initial estimate of
agent $i$ (which can be random).
The steps in~(\ref{eqn:initproj})--\eqref{eqn:algo4batch}
are the successive (random) projections on the sets $\Xc^{\Omega^1_i(k)},\ldots,\Xc^{\Omega^b_i(k)}$
of the point $v_i(k)-\alpha_k\nabla f_i(v_i(k))$.

The algorithm using mini-batches for random projections is of interest when the set $I_i$ is large, i.e.,
the number of constraint set components $\Xc_i^j$, $j\in I_i$, of the set $\Xc_i=\cap_{j\in I_i}\Xc_i^j$ is large.
In such cases, taking several projection steps is beneficial for reducing the infeasibility of the iterates
$x_i(k)$ with respect to the set $\Xc_i$.  More concretely, if each set $\Xc_i$ is the intersection of about $10^4$ simpler sets, then one sample of these sets will render a poor approximation of the true set $\Xc_i$, whereas 100 samples will provide a better approximation of the set.
Let $\check{x}$ be a point in the feasible set $\Xc$.
If just one sample is considered at each iteration, by the non-expansive projection property (Lemma \ref{lem:proj}),
the distance between the next iterate and a point in $\Xc$ can be estimated as:
\[
\|x_i(k+1)-\check{x}\| = \|\psi_i^1(k)-\check{x}\| \le \|\psi_i^0(k)-\check{x}\|,
\]
whereas if 100 samples are considered for projections,
\begin{align*}
\|x_i(k+1)-\check{x}\| = \|\psi_i^{100}(k)-\check{x}\|\le \ldots \\
\le \|\psi_i^1(k)-\check{x}\| \le \|\psi_i^0(k)-\check{x}\|,
\end{align*}
which may yield a larger infeasibility reduction.

For the algorithm using random mini-batch projections, we have the following convergence result.
\begin{proposition} \label{prop:p2}
Let Assumptions \ref{assume:f}-\ref{assume:ds} hold, and
let the stepsize satisfy $\sum_{k=0}^{\infty} \alpha_k = \infty$ and
$\sum_{k=0}^{\infty} \alpha_k^2 < \infty$.
Assume that problem~\eqref{eqn:prob} has a nonempty optimal set $\Xc^*$.
Then, the iterates $\{x_i(k)\}$, $i\in V$, produced by
the method~(\ref{eqn:algo3batch})-(\ref{eqn:algo4batch}) converge
to some random point in the optimal set $\Xc^*$ almost surely, i.e.,
for some random vector $x^\star\in \Xc^*$,
\[
\lim_{k \rightarrow \infty} x_i(k) = x^\star \qquad \text{for all}\quad i \in V \ a.s.
\]
\end{proposition}

\begin{proof}
The proof of this result is similar to that of Proposition~\ref{prop:as}. It requires some adjustments
of Lemma~\ref{lemma:key} and Lemma~\ref{lem:disagree}. The proof with these adjustments is
provided in Appendix~\ref{app:p2proof}.
\end{proof}

\section{Application - Distributed Support Vector Machines (DrSVM) \label{sec:DrSVM}}
In this section, we apply our DRP algorithm and its mini-batch variant to
Support Vector Machines (SVMs). We provide a brief introduction to SVMs
in Subsection~\ref{subsec:SVMs}, while in Subsection~\ref{subsec:sim}
we report our numerical results on some data sets that are generously made available by Thorsten Joachims.

\subsection{Support Vector Machines}\label{subsec:SVMs}
Support Vector Machines (SVMs) are popular classification tools with a strong theoretical background.
Given a set of $n$ example-label pairs $\{(a_j,b_j)\}_{j=1}^n$, $a_j \in \mathbb{R}^d$ and $b_j \in \{+1,-1\}$,
we need to find a vector $x = [y^T~\boldsymbol{\xi}^T]^T \in \mathbb{R}^{d+n}$ that solves the following optimization problem (a bias term is included in $y$ for convenience):
\begin{align}\label{eqn:nonsep}
\min_{y,\boldsymbol{\xi}} {}& f(y,\boldsymbol{\xi}) = \frac{1}{2}\|y\|^2 + C\sum_{j=1}^n \xi_j\\
\text{s.t. } {}& b_j \langle y, a_j \rangle \geq 1 - \xi_j,~
\xi_j \geq 0, \hbox{ for all } j \in \{1,\ldots,n\}. \nonumber
\end{align}
Here, we use slack variables $\xi_j$, for $j = 1, \ldots, n$, to consider linearly non-separable cases as well.
If the optimal solution $(y^*,\boldsymbol{\xi}^*)$ to this problem exists, the solution $y^*$ is the maximum-margin separating hyperplane \cite{Vapnik:1995}.


For applying DRP to problem (\ref{eqn:nonsep}),
we can define $f_i$ and $\mathcal{X}_i$, as follows:
\[
f_i(x) = \frac{1}{2m}\|y\|^2 + C\sum_{j \in I_i} \xi_j,
\]
\[
\mathcal{X}_i = \{x \in \mathbb{R}^{d+n}~|~ b_j\langle y,a_j\rangle
\geq 1 - \xi_j,~ \xi_j \geq 0,~  \forall j \in I_i\}.
\]
where $I_i$ is a set of indices such that $\cup_{i=1}^m I_i = \{1,\ldots,n\}$, $I_i \cap I_j = \emptyset$ for $i \neq j$ and $j \in I_i$ if and only if $\mathcal{X}_i$ contains inequalities associated with the data $(a_j,b_j)$.
Note that each set
$\mathcal{X}_{i}^{j}=\{x \in \mathbb{R}^{d+n}~|~ b_j\langle y,a_j\rangle
\geq 1 - \xi_j,~ \xi_j \geq 0\}$ is the intersection of two halfspaces,
the projection onto which can be computed in a few steps (see Appendix~\ref{app:proj}).

\subsection{Simulations}  \label{subsec:sim}
In the section, we perform some experiments with our DRP algorithm.
We refer to our DRP algorithm applied on SVMs as \textsf{DrSVM}.
The purpose of the experiments is to verify the convergence and to show in how many iterations the proposed method can actually arrive at consensus in distributed settings.
We use the DRP algorithm in~(\ref{eqn:algo3})-(\ref{eqn:algo4}) and
its variant in~(\ref{eqn:algo3batch})-(\ref{eqn:algo4batch}) with
the stepsize $\alpha_k = \frac{1}{k+1}$ for $k \ge 0$.
We vary the number of batches $b$ as 1, 100 or 1000 to observe the different convergence speed,
where $b=1$ corresponds to the algorithm in~(\ref{eqn:algo3})-(\ref{eqn:algo4}).
To show the effect of connectivity, we compare two different time invariant network topologies, i) a completely connected graph (clique) and ii) a 3-regular expander graph. The 3-regular expander graph is a sparse graph that has strong connectivity with every node having degree 3.

We use 3 text classification data sets for our experiments. The data sets were kindly provided by Thorsten Joachims (see \cite{Joachims:2006} for their descriptions).
Table \ref{tbl:stat} lists the statistics of the data sets.
All of the data sets are from binary document classification.
Since the data sets used here are very unlikely separable,
we use the formulation (\ref{eqn:nonsep}) with $C = 1$.
In each experiment the number $n$ of constraints is divided among the agents equally
(if $n$ is not divisible by $m$, the $m$-th agent gets the remainder).
To estimate the generalization (or testing) performance, we split the data and use 80\% for training and 20\% for testing.

\textsf{DrSVM} is implemented with C/C++ and all experiments were performed on a 64-bit machine running Fedora 16 with an Intel Core 2 Quad Processor Q9400 and 8G of RAM.
The experiments are not performed on a real networked environment so we do not consider delays and node/link failures that may exist in networks.

\begin{table}
\centering
\caption{
The statistics of three text classification data sets: $n$ is the number of examples and $d$ is the number of features. $s$ represents the sparsity of data.
\label{tbl:stat}
}
\begin{tabular}{|c||ccc|}\hline
\multirow{2}{*}{Data set}& \multicolumn{3}{|c|}{Statistics} \\
& $n$ & $d$ & $s$ \\\hline
\texttt{astro-ph} & 62,369 & 99,757 & 0.08\% \\\hline
\texttt{CCAT} & 804,414 & 47,236 & 0.16\% \\\hline
\texttt{C11} & 804,414 & 47,236 & 0.16\% \\\hline
\end{tabular}
\end{table}

For stopping criteria, we first run a centralized random incremental projection \cite{AN2011} on the 80\% training set with $b=1$
until the relative error of objective values in two consecutive iterations is less than 0.001. i.e.,
\[
|f(x(k)) - f(x(k+1))|/f(x(k)) < 0.001.
\]
We then measure the test accuracy of the final solution on the remaining 20\% test set,
which will become the target test accuracy $t_{acc}$.
For experiments in the distributed setting, we measure the test accuracy of every agent's solution at
the end of every iteration.
If every solution at certain iteration satisfies the target value $t_{acc}$, we conclude that the agents arrived
at a consensus and the algorithm converged.
The maximum number of iterations in each simulation is limited to 20,000.

Table \ref{tbl:res} shows the results.
As we do more projections per iteration, the total number of iterations required for convergence is less, regardless of the number of agents.
For the given stopping criteria, it seems that less iterations are needed for \textsf{DrSVM} to converge as the number of agents increases.
We can also observe the effect of network connectivity.
When all the other parameters ($m$ and $b$) are the same, for most of the cases,
the number of iterations required for the 3-regular expander graph to converge is greater or equal
to that for the clique.

The table reports the number of iterations required for all the agents to achieve the target test accuracy.
Therefore, the total number of projections is \textit{at most} the number of iterations times $m$ times $b$.
This is because no projection is required if the current estimate is already in the selected constraint component.
For example, the total number of projections for \texttt{astro-ph} with $m=6$ and $b=100$ is at most
$4,800 (= 8\times6\times100)$.

The runtime (or the number of calculations) of the algorithm is not only proportional to the number of projections, but also to the number of gradient updates.
For example, for \texttt{astro-ph} with $m=6$ and $b=1$, the total number of projections is $4,170(=695\times6\times1)$, while the total number of gradient updates is $4,170 (=695\times6)$.
For the same example with  $m=6$ and $b=100$,
the total number of projections is $4,800 (=8\times6\times100)$,
but the total number of gradient updates is only $48 (=8\times6)$.
In any case, the numbers are much smaller than the number 62,369 of the training data points.
This shows that \textsf{DrSVM} can quickly find a good quality solution before examining the training samples even once.

To show the convergence (and consensus) of the algorithm, we plot in Figure \ref{fig:astro} the objective value $f(x)$ of centralized random projection (CRP) and DRP with 10 agents for example \texttt{astro-ph}.
Note that we plot the convergence of the objective value instead of the solution.
This is because CRP and DRP may converge to different optimal points as the problem (\ref{eqn:nonsep}) may not have a unique optimal solution.
For Figure \ref{fig:astro}(a) and \ref{fig:astro}(b), we applied the random projection once and 100 times per iteration, respectively.
From the figures, we can observe that the objective values of CRP and the 10 agents in DRP are almost identical. The final objective of Figure \ref{fig:astro}(b) seems smaller than that of \ref{fig:astro}(a). This is because the stepsize at iteration 1000 is too small.

\begin{center}
\begin{figure}[t]
\subfigure[$b=1$]{
\includegraphics[scale=0.4]{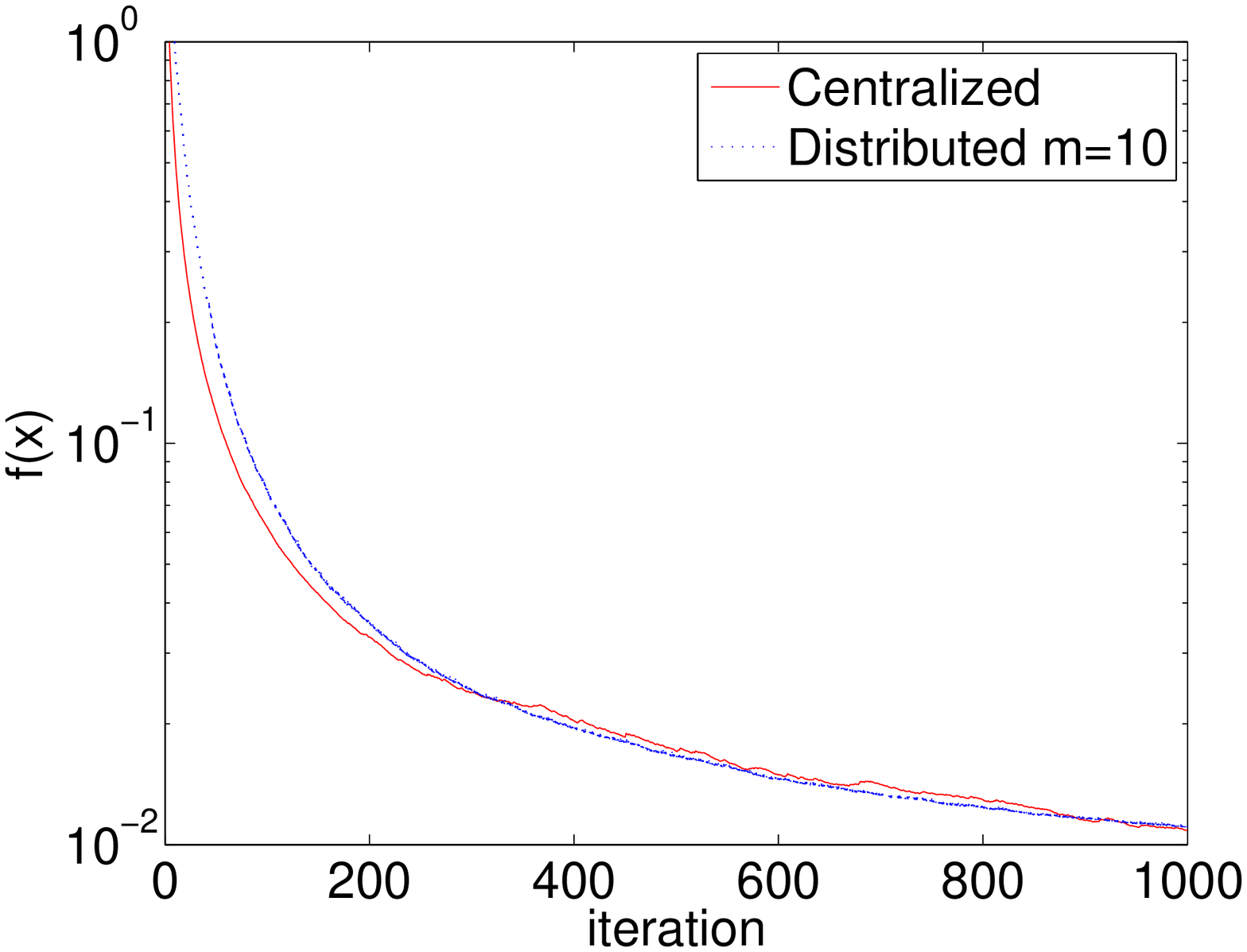}
}
\subfigure[$b=100$]{
\includegraphics[scale=0.4]{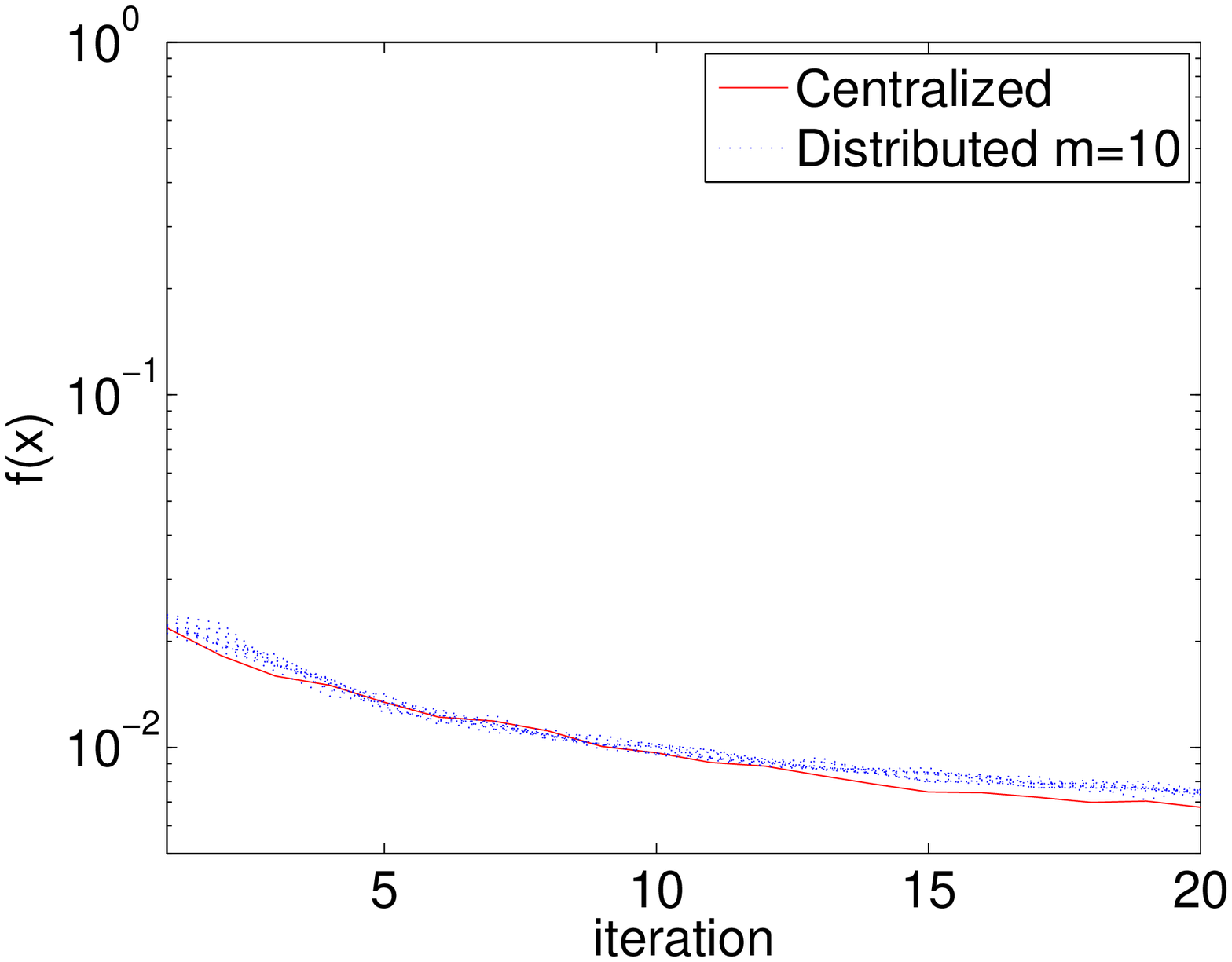}
}
\caption{$f(x)$ vs iteration on \texttt{astro-ph} with 10 agents when batch size $b$ is 1 and 100.  \label{fig:astro}}
\end{figure}
\end{center}

\begin{table*}
\centering
\caption{
The results of \textsf{DrSVM} with two different graph topologies (clique and 3-regular expander graph) and three different numbers of agents ($m=2,6,10$): $t_{acc}$ is the target test accuracy and
$b$ is the number of projections per iteration.
The table shows the number of iterations for all agents to reach the target test accuracy,
where `-' indicates that the algorithm did not converge within the 20,000 maximum iteration limit.
\label{tbl:res}
}
\begin{tabular}{|c|c||c|c|c|c||c|c|}\hline
\multirow{2}{*}{Data set} & \multirow{2}{*}{$t_{acc}$} & \multirow{2}{*}{$b$} & \multirow{2}{*}{$m=2$} & \multicolumn{2}{|c||}{Clique} & \multicolumn{2}{|c|}{3-regular expander}\\\cline{5-8}
& & & & $m=6$ & $m=10$  & $m=6$& $m=10$\\\hline
\multirow{3}{*}{\texttt{astro-ph}}& \multirow{3}{*}{0.95} & 1 & 1,055  & 695  & 697  & 695  & -  \\
& & 100 & 11  & 8  & 11  & 11  & 11   \\
& & 1000 & 2  & 2  & 2  & 2  & 2  \\\hline
\multirow{3}{*}{\texttt{CCAT}} & \multirow{3}{*}{0.91} & 1 & 752  & 511  & 362  & 517  & -  \\
& & 100 & 11  & 10  & 8  & 10  & 8  \\
& & 1000 & 2  & 3  & 2  & 3  & 3  \\\hline
\multirow{3}{*}{\texttt{C11}} & \multirow{3}{*}{0.97} & 1 & 1,511  & 1,255  & 799  & 1,226 & -  \\
& & 100 & 16  & 17 & 12  & 17  & 15  \\
& & 1000 & 2  & 2  & 2  & 2 & 2  \\\hline
\end{tabular}
\end{table*}

\section{Conclusions  \label{sec:con}}
We have proposed and analyzed a distributed gradient algorithm with random incremental projections for a network of agents with time-varying connectivity.
We considered the most general cases, where each agent has a unique and different objective and constraint.
The proposed algorithm is applicable to problems where the whole constraint set is not known in advance but its component is revealed in time, or where the projection onto the whole set is computationally prohibitive.
We have established almost sure convergence of the algorithm when the objective is convex under typical assumptions.
Also, we have provided a variant of the algorithm using a mini-batch of consecutive projections and established its convergence in almost sure sense.
Experiments on three text classification benchmarks using SVMs were performed to verify the performance of the proposed algorithm.

Future work includes some extensions of the distributed model proposed here.
First, we have assumed the gradients evaluated have no errors. We can consider the effects of stochastic gradient errors in the future analysis.
Second, more robust algorithms can be developed to also handle asynchronous networks with communication delays, noise and/or failures in links/nodes.
Third, an implementation of a real parallel computing environment will be needed to
handle large-scale data sets.

\appendix

\subsection{Proof of Proposition~\ref{prop:p2}}\label{app:p2proof}
We construct the proof by adjusting the result of Lemma~\ref{lem:first},
and by verifying that Lemma~\ref{lemma:key} and Lemma~\ref{lem:disagree} apply to the mini-batch
variant of the DRP method. The basic insight that guides the proof is that the operation
of successive projections on components $\Xc_i^j$ of the set $\Xc_i=\cap_{j\in I_i}\Xc_j^j$
remains a non-expansive operation
with respect to points that belong to the set $\Xc_i$, as well as
with respect to the points in the intersection set $\Xc=\cap_{i=1}^m\Xc_i$.

\subsubsection{Basic Iterate Relation for Mini-Batch Algorithm}
For the iterates generated by the mini-batch random projection algorithm
in~\eqref{eqn:algo3batch}--\eqref{eqn:algo4batch}, we have the following basic result.
\begin{lemma} \label{lemma:first-batch}
Let Assumption~\ref{assume:f} hold.
 Then, for any $\check{x}\in \Xc$, and for all $i\in V$ and all $k\ge0$,
\begin{align*}
\|{}&x_i(k+1)-\check{x}\|^2
\le  (1+A_{\tau}\alpha_k^2)\|v_i(k)-\check{x}\|^2 \nonumber\\
{}&  - 2\alpha_k(f_i(z)-f_i(\check{x}))- \frac{3}{4}\|\psi_i^1(k)- v_i(k)\|^2\\
{}& + \left(\frac{3}{8\tau}+2\alpha_k L\right)\|v_i(k)-z\|^2
+ B_{\tau}\alpha^2_kG_f^2,
\end{align*}
where $A_{\tau} = 8L^2 + 16\tau L^2$, $B_{\tau} = 8\tau+8$ and $\tau >0 $ is arbitrary.
\end{lemma}
\begin{proof}
By using the non-expansiveness property of projection operation (Lemma~\ref{lem:proj}(a)),
we have for arbitrary $\check{x}\in \Xc$ (since $\Xc\subseteq \Xc_i^j$ for all $j\in I_i$),
and for all $i\in V$ and $k\ge0$,
\begin{align}\label{eqn:b1}
\|x_i(k+1){}&-\check{x}\|\le \|\psi_i^{b-1}(k)-\check{x}\|\nonumber\\
{}&\le\cdots\le \|\psi_i^1(k)-\check{x}\|.
\end{align}
The intermediate iterate $\psi^1_i(k)$ is just obtained after one projection step,
\[\psi_i^1(k)= \mathsf{\Pi}_{\mathcal{X}_i^{\Omega^1_i(k)}} [v_i(k)-\alpha_k\nabla f_i(v_i(k))],\]
so it satisfies Lemma~\ref{lem:first} with $y=\psi^1_i(k)$, $\mathcal{Y}=\Xc_i^{\Omega_i^1(k)}$,
$x=v_i(k)$, $\alpha=\alpha_k$, and $\phi=f_i$. Thus, we have
for any $\check{x} \in \mathcal{X}$ and $z \in \mathbb{R}^d$,
\begin{align}\label{eqn:b2}
{}&\|\psi_i^1(k)-\check{x}\|^2
\le  (1+A_{\tau}\alpha_k^2)\|v_i(k)-\check{x}\|^2\nonumber\\
{}&  - 2\alpha_k(f_i(z)-f_i(\check{x}))- \frac{3}{4}\|\psi_i^1(k)-v_i(k)\|^2\nonumber\\
{}& +\hspace{-0.5mm} \left(\hspace{-0.5mm}\frac{3}{8\tau}\hspace{-0.5mm}+\hspace{-0.5mm}2\alpha_k L\hspace{-0.5mm}\right)\hspace{-0.5mm}\|v_i(k)\hspace{-0.5mm}-\hspace{-0.5mm}z\|^2
\hspace{-0.5mm}+\hspace{-0.5mm} B_{\tau}\alpha^2_k\|\nabla f_i(\check{x})\|^2.
\end{align}
From~\eqref{eqn:b1} and~\eqref{eqn:b2}, by using the gradient boundedness property of
Assumption~\ref{assume:f}(d), we obtain the stated relation.
\end{proof}

\subsubsection{Conditional Expectation Relation for Mini-Batch Algorithm}
The convergence proof of Proposition~\ref{prop:p2}
requires a relation for the iterates involving expectations with respect to the past history
of the method. For this, we need to define a relevant $\sigma$-algebra.
We let $\mathcal{\tilde F}_k$ be the $\sigma$-algebra generated by the entire history of
the algorithm up to time $k-1$ inclusively. Thus, $\mathcal{\tilde F}_k$ includes
the realizations of all the random variables but not the realizations of the indices
$\Omega_i^1(k),\ldots,\Omega_i^b(k)$ at time $k$. Specifically, it is given by for all $k \ge 1$,
\begin{align*}
\tilde \Fc_k = & \{x_i(0), i \in V\}  \\
& \cup\{\Omega_i^r(\ell); 0\le \ell \le k-1, 1\le r\le b, i \in V\}
\end{align*}
where $\tilde\Fc_0 = \{x_i(0), i \in V\}$.

Now, with this definition of the $\sigma$-algebra, we have the following result.

\begin{lemma} \label{lemma:exp-batch}
Let Assumptions~\ref{assume:f} and~\ref{assume:c} hold. Then, almost surely
for any $\check{x}\in \Xc$, and for all $i\in V$ and all $k\ge0$,
\begin{align*}
\mathsf{E}\Big[\|{}&x_i(k+1)-\check{x}\|^2\mid\tilde \Fc_k\Big]
\le  (1+A\alpha_k^2)\|v_i(k)-\check{x}\|^2\\
{}&  - 2\alpha_k(f_i( z_i(k)) - f_i(\check{x}))\\
{}& -\left(\frac{3}{8c}-2\alpha_k L\right)\dist^2(v_ix(k),\Xc)
+ B\alpha^2_kG_f^2.
\end{align*}
where  $z_i(k)=\mathsf{\Pi}_{\mathcal{X}}[v_i(k)]$,
$A = 8L^2 + 16c L^2$, $B = 8c+8$, and $c$ is from Assumption~\ref{assume:c}.
\end{lemma}

\begin{proof}
By letting $z=z_i(k)$ and $\tau=c$ in Lemma~\ref{lemma:first-batch},
we obtain
\begin{align*}
{}&\mathsf{E}\Big[\|x_i(k+1)-\check{x}\|^2\mid\tilde \Fc_k\Big]
\le  (1+A\alpha_k^2)\|v_i(k)-\check{x}\|^2\\
{}&  - 2\alpha_k(f_i( z_i(k) ) \hspace{-0.5mm}- \hspace{-0.5mm}f_i(\check{x}))
 \hspace{-0.5mm}-\hspace{-0.5mm} \frac{3}{4}\hspace{-0.5mm}\EXP{\|\psi_i^1(k)- v_i(k)\|^2\mid \tilde \Fc_k}\\
{}& + \left(\frac{3}{8c}+2\alpha_k L\right)\|v_i(k)-z_i(k)\|^2
+ B\alpha^2_kG_f^2,
\end{align*}
where $A = 8L^2+16cL^2$ and $B=8c+8$.

Since $\psi_i^1(k)\in\Xc_i^{\Omega^1_i(k)}$, by the projection property we have
$\|\psi_i^1(k)- v_i(k)\|^2\ge\|\mathsf{\Pi}_{\mathcal{X}^{\Omega^1_i(k)} }[v_i(k)]-v_i(k)\|^2.$
Then,
\begin{align*}
\mathsf{E}{}&\Big[\|\psi_i^1(k)- v_i(k)\|^2\tilde \Fc_k\Big]\\
{}&\ge\EXP{\|\mathsf{\Pi}_{\mathcal{X}^{\Omega^1_i(k)} }[v_i(k)]-v_i(k)\|^2 \mid \tilde \Fc_k}\\
{}&=\EXP{\|\mathsf{\Pi}_{\mathcal{X}^{\Omega^1_i(k)} }[v_i(k)]-v_i(k)\|^2\mid v_i(k)}.
\end{align*}
Furthermore, by Assumption~\ref{assume:c} we have
\[\EXP{\|\mathsf{\Pi}_{\mathcal{X}^{\Omega^1_i(k)} }[v_i(k)]\hspace{-0.5mm}-\hspace{-0.5mm}v_i(k)\|^2\hspace{-0.5mm}\mid\hspace{-0.5mm} v_i(k)}
\hspace{-0.5mm}\ge\hspace{-0.5mm} \frac{1}{c}\dist^2(v_i(k),\Xc).\]
The preceding relations and $\dist(v_i(k),\Xc)=\|v_i(k) -z_i(k)\|$ yield the desired relation.
\end{proof}

\subsubsection{Lemma~\ref{lemma:key} and Lemma~\ref{lem:disagree} hold}
Using Lemma~\ref{lemma:exp-batch}, we argue that the results
of Lemma~\ref{lemma:key} and Lemma~\ref{lem:disagree} apply to the mini-batch random
projection algorithm.
\begin{claim}
Lemma~\ref{lemma:key} holds for the iterates generated by method~\eqref{eqn:algo3batch}--\eqref{eqn:algo4batch}.
\end{claim}

\begin{proof}
By letting $\check{x}=\mathsf{\Pi}_{\mathcal{X}}[v_i(k)]$ in Lemma~\ref{lemma:exp-batch},
and noting that $\|x_i(k+1)-\mathsf{\Pi}_{\mathcal{X}}[v_i(k)]\|\ge\dist(x_i(k+1),\Xc)$
and $\|v_i(k) -\mathsf{\Pi}_{\mathcal{X}}[v_i(k)]\|=\dist(v_i(k),\Xc)$,
we obtain
\begin{align*}
\mathsf{E}{}&\Big[\dist^2(x_i(k+1),\Xc)\mid\tilde \Fc_k\Big]
\hspace{-0.5mm}\le \hspace{-0.5mm} (1\hspace{-0.5mm}+\hspace{-0.5mm}A\alpha^2)\dist^2(v_i(k),\Xc)\\
{}& -\left(\frac{3}{8c}-2\alpha_k L\right)\dist^2(v_i(k),\Xc)+ B\alpha^2_kG_f^2,
\end{align*}
which is the same as relation~\eqref{eqn:e03} within the proof of Lemma~\ref{lemma:key}.
The rest of the proof of Lemma~\ref{lemma:key} holds exactly as given, and the result of
Lemma~\ref{lemma:key} remains valid.
\end{proof}

\begin{claim}
Lemma~\ref{lem:disagree} holds for the iterates generated by
method~\eqref{eqn:algo3batch}--\eqref{eqn:algo4batch}.
\end{claim}
\begin{proof}
Define $e_i(k)=x_i(k+1)-v_i(k)$ and
$z_i(k) \triangleq \mathsf{\Pi}_{\mathcal{X}}[v_i(k)]$.
Now, consider $\|e_i(k)\|$ for which we have
\[\|e_i(k)\|\le\|x_i(k+1) - z_i(k)\| +\|z_i(k) - v_i(k)\|.\]
The non-expansiveness projection property and the fact $z_i(k)\in\Xc\subset\Xc_i^{\Omega_i^r(k)}$,
for all $r=1,\ldots,b$, and any realization of these sets imply
\begin{align*}
\|{}&x_i(k+1)-z_i(k)\|\\
{}&\le \|\psi_i^{b-1}(k)-z_i(k)\|\le\cdots\le\|\psi^1_i(k)-z_i(k)\|\\
{}&\le \|v_i(k)-\alpha_k\nabla f_i(v_i(k)) - z_i(k)\|.
\end{align*}
Therefore
\[\|\hspace{-0.5mm}e_i(k)\hspace{-0.5mm}\|\hspace{-0.5mm}\le\hspace{-0.5mm} \|v_i(k)-\alpha_k\hspace{-0.5mm}\nabla\hspace{-0.5mm} f_i(v_i(k)) - z_i(k)\|+\|z_i(k) - v_i(k)\|,\]
which is the same as the first inequality in~\eqref{eqn:e_final} within the proof of
Lemma~\ref{lem:disagree}. The rest of the proof of that lemma holds in verbatim, and the result follows.
\end{proof}

\subsubsection{Details of the Proof of Proposition~\ref{prop:p2}}
We now connect the preceding results and provide the proof of Proposition~\ref{prop:p2}.
Starting from the relation in Lemma~\ref{lemma:exp-batch}, after summing over all $i\in V$,
we can see that almost surely for all $\check{x}\in \Xc$ and all $k\ge0$,
\begin{align*}
\mathsf{E}{}&\Big[\sum_{i=1}^m \|x_i(k+1)-\check{x}\|^2\mid\tilde \Fc_k\Big] \\
{}&\le  (1+A\alpha_k^2)\sum_{i=1}^m\|v_i(k)-\check{x}\|^2\\
{}&  - 2\alpha_k\sum_{i=1}^m(f_i( z_i(k)) - f_i(\check{x}))\\
{}& -\left(\frac{3}{8c}-2\alpha_k L\right)\sum_{i=1}^m\|v_i(k)- z_i(k)\|^2
+ mB\alpha^2_kG_f^2,
\end{align*}
where $z_i(k)=\mathsf{\Pi}_{\mathcal{X}}[v_i(k)]$.

Now, the same as in the proof of Proposition~\ref{prop:as}, using the properties of the matrices $W(k)$
and the convexity of the squared-norm function (see~\eqref{eqn:norm}),
we can show that
\[\sum_{i=1}^m\|v_i(k)-\check{x}\|^2 \le\sum_{j=1}^m\|x_j(k)-\check{x}\|^2.\]
Also, using verbatim arguments, we can show that
\begin{align*}
\sum_{i=1}^m{}&(f_i( z_i(k)) - f_i(\check{x}))\\
{}&\ge -2G_f\sum_{i=1}^m\|v_i(k)- \bar v(k)\|+\left(f(\bar z(k)-f(\check{x})\right),
\end{align*}
where  $\bar z(k)= \frac{1}{m}\sum_{\ell=1}^m z_\ell (k)$ and  $\bar v(k)=\frac{1}{m}\sum_{\ell=1}^m v_\ell (k)$.
Under the conditions of Proposition~\ref{prop:p2}, we have $\alpha_k\to0$.
Choosing $\bar k$ large enough so that $2\alpha_kL \le \frac{3}{8c} $ for all $k\ge \bar k$,
we have
\[-\left(\frac{3}{8c}-2\alpha_k L\right)\sum_{i=1}^m\|v_i(k)- z_i(k)\|^2\le 0.\]
By combining all the preceding relations, we obtain
almost surely for all $\check{x}\in \Xc$ and all $k\ge \bar k$,
\begin{align*}
\mathsf{E}\Bigg[{}&\sum_{i=1}^m\|x_i(k+1)-\check{x}\|^2\mid\tilde \Fc_k\Bigg] \\
\le {}& (1+A\alpha_k^2)\sum_{i=1}^m\|x_i(k)-\check{x}\|^2\\
{}&  - 2\alpha_k( f(\bar z (k)) - f(\check{x})) \\
 {}& +4\alpha_k G_f\sum_{i=1}^m\|v_i(k)- \bar v(k)\|
+ mB\alpha^2_kG_f^2.
\end{align*}
Letting $\check{x}= x^*$ for an arbitrary optimal solution $x^*\in\Xc^*$, from the preceding relation
we arrive at relation~\eqref{eqn:final} in the proof of Proposition~\ref{prop:as}.
From relation~\eqref{eqn:final} onward,
the proof of Proposition~\ref{prop:as} holds verbatim, and the stated almost sure convergence of
the mini-batch method follows.

\subsection{Projection onto the Intersection of Two Half-spaces} \label{app:proj}
Given $v \in \mathbb{R}^d$, we are interested in solving the following optimization problem.
\begin{align}\label{eqn:proj}
\min_{w \in \mathbb{R}^d} {}& ~\frac{1}{2}\|w-v\|^2\\
\text{s.t. } {}& ~\langle a,w \rangle \leq b,~ w_i \geq 0,\nonumber
\end{align}
where $a \in \mathbb{R}^d$, $b \in \mathbb{R}$ and $w_i$ is the $i$-th component of the vector $w$.

The two half-spaces divide the $\mathbb{R}^d$ space into four parts. Therefore, there are only four cases to consider.
\begin{enumerate}
\item $\langle a,v \rangle \leq b$ and $v_i \geq 0$.\\
In this case, $v$ is already in the intersection and $w = v$.
\item $\langle a,v\rangle > b$ and $v_i < 0$. \\
In this case, $v$ is projected onto the intersection of the two hyperplanes
$\{w\mid\langle a,w \rangle=b\}$ and $\{w\mid w_i=0\}$.
Finding such a projection is equivalent to solving the following optimization problem:
\begin{align}\label{eqn:proj2}
\min_{w \in \mathbb{R}^d} {}& ~\frac{1}{2}\|w-v\|^2\\
\text{s.t. } {}& ~\langle a,w \rangle = b,~ w_i = 0.\nonumber
\end{align}
The Lagrangian of the problem (\ref{eqn:proj2}) is
\[
\mathcal{L}(w,\theta,\zeta) = \frac{1}{2}\|w-v\|^2 + \theta\left(\sum_{j=1}^d a_jw_j-b\right) + \zeta w_i,
\]
where $\theta, \zeta \in \mathbb{R}$ are Lagrange multipliers.
Differentiating the Lagrangian and setting it to zero gives the optimality condition,
\[
w_i^*-v_i+a_i\theta^*+\zeta^* = 0,
\]
\[
w_j^*-v_j+a_j\theta^* = 0 ,\quad \text{for } j \neq i.
\]

From the primal feasibility, we have the following relations:
\[
w_i^* = 0 \implies \zeta^* =  v_i-a_i\theta^*,
\]
\[
\sum_{j=1}^n a_jw_j^* = \sum_{j\neq i}a_jw_j^* = \sum_{j\neq i}a_j(v_j - a_j\theta^*) = b
\]
\[
\implies \quad\theta^* = \frac{\sum_{j\neq i}a_jv_j-b}{\sum_{j\neq i}a_j^2}.
\]
Therefore, the projection is given by
\begin{equation*} \label{eqn:cross}
w_j^* = \left\{
\begin{array}{ll}
0& \text{ if } j = i,\\
v_j - a_j\theta^*& \text{ otherwise.}
\end{array} \right.
\end{equation*}

Let $w^* = [w_1^*,\ldots,w_d^*]^T$.

\item $\langle a,v\rangle > b$ and $v_i \geq 0$.\\
In this case, $v$ will be projected either onto the hyperplane
$\{w\mid \langle a,w \rangle = b\}$ or onto the intersection of the two hyperplanes
$\{w\mid \langle a,w \rangle=b\}$ and $\{w\mid w_i=0\}$.
Let $\hat{w}$ be the projection of $v$ onto $\{w\mid \langle a,w \rangle=b\}$, i.e.,
\[
\hat{w} = v - \left(\frac{\langle a,v \rangle-b}{\|a\|^2}\right)a.
\]
The projection of $v$ in this case is given by
\begin{equation*}
w = \left\{
\begin{array}{ll}
\hat{w} & \text{ if } \hat{w}_i \geq 0,\\
w^*& \text{ otherwise.}
\end{array} \right.
\end{equation*}

\item $\langle a,v \rangle \leq b$ and $v_i < 0$.\\
Let $\hat{w}$ be the projection of $v$ onto the hyperplane $\{w\mid w_i=0\}$, i.e.,
\[
\hat{w} = v - \left(v_i-b\right)e_i,
\]
where $e_i \in \mathbb{R}^d$ is the vector whose $i$-th component is one and
all the other components are zero. Then, the projection of $v$ is given by
\begin{equation*}
w = \left\{
\begin{array}{ll}
\hat{w} & \text{ if } \langle a,\hat{w}\rangle \leq b,\\
w^*& \text{ otherwise.}
\end{array} \right.
\end{equation*}

\end{enumerate}

\bibliographystyle{ieeetran}
\bibliography{soomin001}

\end{document}